\documentclass{scrartcl}
\usepackage{amsfonts, amssymb, amsthm}
\usepackage{tikz-cd}

\usepackage[T1]{fontenc}
\usepackage[utf8]{inputenc}
\usepackage{mathtools,bookmark,cleveref,enumitem,mleftright}
\usepackage{graphicx, shuffle}
\usepackage{subfig, cite}

\usepackage[algo2e]{algorithm2e}

\definecolor{algoColorKeyword}{named}{cyan}
\definecolor{algoColorComment}{named}{olive}

\SetAlFnt{\normalfont\ttfamily}

\SetKwFunction{FRecurs}{FnRecursive}%

\usepackage[linecolor=white,backgroundcolor=white,bordercolor=white,textsize=tiny]{todonotes}
\let\todon\todo
\renewcommand{\todo}[1]{\todon[inline,color=green!40]{\color{magenta}{#1}}}

\usepackage{listings}

\graphicspath{{../}}

\ifpdf
\hypersetup{
  pdftitle={Tropical time series, iterated-sums signatures and quasisymmetric functions},
  pdfauthor={J. Diehl, K. Ebrahimi-Fard, and N. Tapia}
}
\fi

\setlist[enumerate]{leftmargin=.5in}
\setlist[itemize]{leftmargin=.5in}

\newcommand{\evaluatedAt}[1]{\,\raisebox{-.5em}{$\vert_{#1}$}}

\newtheorem{theorem}{Theorem}[section]
\newtheorem{lemma}[theorem]{Lemma}
\newtheorem{proposition}[theorem]{Proposition}

\newtheorem{corollary}[theorem]{Corollary}

\newtheorem{remark}[theorem]{Remark}
\newtheorem{example}[theorem]{Example}

\theoremstyle{definition}
\newtheorem{definition}[theorem]{Definition}

\def\letter#1{\mathtt{#1}}
\def\llangle{\mathopen{\langle\mkern-3.5mu\langle}}
\def\rrangle{\mathclose{\rangle\mkern-3.5mu\rangle}}
\def\lbracket{\mathopen{[\mkern-3mu[}}
\def\rbracket{\mathclose{]\mkern-3mu]}}

\DeclareMathOperator{\id}{id}
\newcommand*{\R}{{\mathbb R}}

\newcommand\N{\mathbb N}

\newcommand\zero{\boldsymbol{0}}
\newcommand\one{\boldsymbol{1}}
\newcommand\ISS{\operatorname{ISS}}
\renewcommand\mp{{ \scalebox{0.4}{$\operatorname{min+}$} }}
\newcommand\maxp{{ \scalebox{0.4}{$\operatorname{max+}$} }}
\newcommand\TSS{\operatorname{ISS}^\mathbb{S}}
\newcommand\w[1]{\letter{#1}}
\newcommand\Z{{\mathbb{Z}}}
\renewcommand\S{{\mathbb{S}}}
\newcommand\s{{\scalebox{0.4}{$\mathbb{S}$}}}
\newcommand\y{{\scalebox{0.4}{$\mathbb{A}$}}}

\newcommand\insertZero{\mathsf{zero}}
\newcommand\bigopluss{\sideset{}{_\s}\bigoplus}
\newcommand\bigodotss{\sideset{}{_\s}\bigodot}
\newcommand\odots{\odot_\s}
\newcommand\opluss{\oplus_\s}
\newcommand\qs\star
\newcommand\emptyWord{\mathbf{e}}
\renewcommand\O{\mathcal{O}}
\newcommand\len{\ell}
\mleftright
\begin{document}

\title{Tropical time series, iterated-sums signatures and quasisymmetric functions}

\author{Joscha Diehl\thanks{Universit\"at Greifswald, Institut f\"ur Mathematik und Informatik, Walther-Rathenau-Str.~47, 17489 Greifswald, Germany.},
Kurusch Ebrahimi-Fard \thanks{Department of Mathematical Sciences, NTNU, 7491 Trondheim, Norway.},
Nikolas Tapia \thanks{Weierstra{\ss}-Institut Berlin, Mohrenstr.~39, 10117 Berlin, Germany\newline Technische Universität Berlin, Str. des 17.~Juni 136, 10623 Berlin, Germany.}}



\maketitle

\begin{abstract}
  Aiming for a systematic feature-extraction from time series,
  we introduce the iterated-sums signature over arbitrary commutative semirings.
  The case of the tropical semiring is a central, and our motivating example.
  It leads to features of (real-valued) time series that are
  not easily available using existing signature-type objects.
  We demonstrate how the signature
  extracts chronological aspects of a time series,
  and that its calculation is possible in linear time.
  We identify quasisymmetric expressions
  over semirings as the appropriate framework for iterated-sums signatures over semiring-valued time series.  
\end{abstract}

\tableofcontents

\newpage
%
%
%
%
%


\section{Introduction}
\label{sec:intro}

Recent developments \cite{DET2020,KBP+2019, KO2019,LLY+2019,toth2020seq2tens} have shown that various forms of \emph{iterated-sum} and \emph{iterated-integral} operations can form a useful component in machine learning pipelines for sequential data.
Originating from the study of (discretized) controlled ordinary differential equations (ODEs) \cite{fliess1981fonctionnelles,Lyons1998}, they are particularly apt to model input-output relations that are well-approximated by dynamical systems \cite{DGE2016,Gray2017}.
In fact, the \emph{iterated-integrals signature} $\mathrm{IIS}(x)$%
\footnote{Also just called the \emph{signature} and denoted with $S(X)$ or $\sigma(X)$.}
of a (smooth enough) multidimensional curve $x = (x^{(1)},\dots,x^{(d)})\colon [0,T] \to \R^d$,
is the solution to a certain \emph{universal} controlled ODE \cite[Proposition 7.8]{FrizVic2010}.
It is universal in the sense
that the solution to \emph{any} other controlled ODE can be well-approximated by a linear expression of the
iterated-integrals signature.
On a more concrete level,
the entries of the ``classical'' iterated-integrals signature 
are real numbers, indexed by words $w=w_1 \dotsm w_k$ in the alphabet $A^{\prime}=\{1,\dotsc,d \}$,
and given as follows
\begin{align}
\label{Ssu}
  \Big\langle \mathrm{IIS}(x)_{0,T}, w \Big\rangle 
  &= \int_{0 \le t_1 \le t_2 \le \cdots \le t_k \leq T}\mathrm dx_{t_1}^{(w_1)} \cdots \mathrm dx_{t_k}^{(w_k)} \\
  &= \int_{0 \le t_1 \le t_2 \le \cdots \le t_k \leq T} \dot x_{t_1}^{(w_1)} \cdots \dot x_{t_k}^{(w_k)}\,\mathrm dt_1 \cdots \mathrm dt_k.\notag
\end{align}
Here and throughout the rest of the article, we use subscript notation to denote the value at (continuous or discrete) time \(t\).

Not all input-output relations are well-modeled by controlled ODEs, though.
As an extreme example, we mention that a controlled ODE does not care about so-called ``tree-like'' excursions
of the driving signal \cite{HL2010}.
The iterated-integrals signature can therefore, for example, not distinguish the following
two one-dimensional curves, $t \in [0,1]$,
\begin{align*}
  t \mapsto 0 \qquad \text{ and } \qquad t \mapsto \sin( 2 \pi t ).
\end{align*}
There are several ways to circumvent this particular problem, e.g. ``lifting'' a one-dimensional curve
to a two-dimensional curve \cite{FHL2016}.
The \emph{iterated-sums signature} (ISS) introduced in \cite{DET2020} (see also \cite{KO2019,toth2020seq2tens}) forgoes this particular problem altogether
and brings the added benefit of working directly with discrete time series
(in order to apply the theory of iterated \emph{integrals} to discrete-time sequential data,
it has to be interpolated to a, say, piecewise linear curve).

But even the ISS cannot ``see'' all aspects of a time series.
Indeed, the ISS is invariant to time warping and hence cannot distinguish
time series run at different speeds (see \Cref{lem:idempotentTimeWarping} for a proper definition).
It turns out that such invariance is often desirable.
The search for such invariants was in fact the starting point of \cite{DET2020},
where it is shown that the ISS contains all polynomial expressions in the time series entries that are invariant to time warping.
There are, however, non-polynomial time-warping invariant functionals as the following example shows (this is well-defined for any time series $z$ that is eventually constant):
\begin{align}
  \label{eq:min}
  (z_1,z_2,\ldots) \mapsto \min_{j} z_j.
\end{align}
Such expressions are ubiquitous in financial mathematics, in particular the pricing of American options \cite{myneni1992pricing}.
Moreover, it is folklore knowledge that such expressions are poorly approximated by polynomials.
This is related to the fact that this functional is \emph{not} well-approximated by (discretized) ODEs.
\begin{quote}
  \emph{The aim of the work at hand is to introduce a systematic feature-extraction method for time series
  that encompasses functionals as the one in \cref{eq:min}.}
\end{quote}

The entry point for our investigation is the observation
that \eqref{eq:min} can be considered as a polynomial expression if one changes the underlying
field of the reals to the tropical (or min-plus) semiring.
This, as well as other semirings, have (a subset of) the reals as the underlying set,
and only the operations of ``addition'' and ``multiplication'' have a different meaning.
As a result, this opens ways to consider real-valued time series under many different lenses.
In particular, it allows us to consider \eqref{eq:min} as part of an iterated-sums signature.
Moreover, semirings whose underlying sets are not given by subsets of the real line enable one to look at time series with values in more general spaces.

After this general motivation, we now present two ways to 
naturally arrive at the signature we introduce in this work,
where the first one makes the remarks above, in particular those concerning the invariant in \eqref{eq:min}, more concrete.
\begin{enumerate}
  \item \textbf{Invariants of a time series.}
Accommodating the discrete nature of the time series $x=(x_1, x_2,\dots, x_N)$, $x_i \in \R^d$, one may consider the discrete analog of \eqref{Ssu}, i.e., the so-called iterated-sums signature $\ISS$ over the reals is defined in terms of the increments $\delta x_i=x_i - x_{i-1}$
\begin{align}
\label{ISSpq}
\Big\langle \ISS(x)_{0,N}, w \Big\rangle = \sum_{0 < i_1 <i_2 <\cdots < i_k \leq N} (\delta x_{i_1})^{w_1} \cdots (\delta x_{i_k})^{w_k}.
\end{align}
Here $w=w_1 \cdots w_k$ is a word over a certain alphabet -- larger than $A'$ from above -- which is adapted to the discrete nature of the summation operation. It was shown in \cite{DET2020} that the map $\ISS(x)$ stores all polynomial invariants to time warping (and translation).

Alternatively, the $\ISS$ can be defined as
\begin{align*}
  \Big\langle \ISS(z)_{p,q}, w \Big\rangle = \sum_{p < i_1 <i_2 <\cdots < i_k \leq q} (z_{i_1})^{w_1} \cdots (z_{i_k})^{w_k},
\end{align*}
and the former definition is obtained by evaluating at increments, $z_i=\delta x_i$.
The latter definition yields an object that is invariant to insertion of $0 \in \R$ into the time series $z=(z_1,z_2,\ldots)$,
and it is this viewpoint that generalizes to arbitrary semirings.
More precisely, in \Cref{sec:tss} we construct a signature over commutative semirings
that is invariant to the insertion of zeros of the semiring.
The underlying mathematical object are quasisymmetric expressions over commutative semirings,
which, to the best of our knowledge, we introduce for the first time in  \Cref{sec:qsym}.
The way back from invariants to inserting zeros to time-warping invariants is not as straightforward
as in the case when considering computations over a field, and we investigate it in Section  \ref{sec:timeWarpingInvariants}.
As we will see, expressions like \eqref{eq:min} will be covered by the resulting theory.

\item \textbf{Cheap chronological information of a time series.}

The importance of convolutional neural networks (CNNs) is hard to overestimate \cite{krizhevsky2012imagenet}.
Their success, in particular in image recognition, is usually attributed to two ingredients:
\begin{enumerate}

  \item
    \emph{weight sharing} which reduces, in comparison to fully connected networks, the amount
    of parameters and hence allows for deeper architectures, and

  \item 
    convolution and its particular \emph{structure}  (usually combined with max-pooling)
    leads to desirable properties with respect to image recognition
    (modeling of receptive fields, approximate translation invariance, etc.).

\end{enumerate}

Although CNNs have been successfully applied in the context of time series data (see \cite{bai2018empirical,FFW+2019} for recent surveys), this does not seem to be based on the inherent properties of sequential data.
In particular, the \emph{structure} of time series is very different from that of images. It is not clear why the receptive-field structure of CNNs captures intrinsically meaningful information of sequential data. Moreover, time series possess a characteristic that images do not: a \emph{chronology}, that is, the values of the series are \emph{ordered} according to the time index.
We explain this using an example.

\begin{example}
  \label{ex:toyExample}
  We consider a very concrete toy example.
  Let the input $x \in \{2,4,8,16\}^4$ consist in sequences of length four in the numbers $2,4,8,16$.
  On this input space we consider a convolutional layer with kernel-size $2$, stride $1$, followed
  by a max-pool with kernel-size $4$.
  For example
  \begin{align*}
    \begin{bmatrix}
      {2} & {4} & {4} & {16}
    \end{bmatrix}
    \mapsto
    \max \{ a_1\cdot 2 + a_2\cdot 4, a_1\cdot 4 + a_2\cdot 4, a_1\cdot 4 + a_2\cdot 16 \},
  \end{align*}
  where $a_1,a_2 \in \R$ are the parameters of the convolutional kernel.
  This network can learn to answer questions of the following type.
  \begin{itemize}
    \item Is there a $16$ in the sequence \emph{somewhere}? (Just as a full-blown CNN on image data can answer the question: Is there a dog \emph{somewhere} in the picture?)
    \item Is there a $2$ \emph{directly} followed by a $16$ somewhere in the sequence?
      (indeed, with $a_1=-1, a_2=1$ one gets output $14$ if and only if the statement is true)
  \end{itemize}

  However, it can \emph{not} answer the question
  \begin{itemize}
    \item Is there a $2$ somewhere and then, \emph{sometime after}, a $16$ in the sequence?
  \end{itemize}

  We believe that \emph{chronological questions} of this type are the relevant questions for time series.
  Note that the following architecture allows us to answer this question
  \begin{align*}
    \begin{bmatrix}
      x_1 & x_2 & x_3 & x_4
    \end{bmatrix}
    \mapsto
    \max \{ a_1\cdot x_{i_1} + a_2\cdot x_{i_2} : i_1 < i_2 \}.
  \end{align*}
  Indeed, again with $a_1=-1, a_2=1$, the output is $14$ if and only if the question is answered positively.
\end{example}

\begin{example}
  The preceeding example was \emph{strictly for illustrative purposes}
  and the chronological question posed there can be answered by other means
  (for example a simple inite state automaton).
  To see how its \emph{conceptual idea} can be applied in a deep learning pipeline though,
  we have set up the github repository \url{https://github.com/diehlj/tropical-time-series},
  which implements the following binary classification problem.

  One sample $x$ is generated as follows (we have chosen the 'pattern' $[2,-3,16]$ of length three now).
  \begin{itemize}
    \item Let $\xi_i \sim N(0,\sigma^2), i=1, \dots, T$ are i.i.d. random variables (``white noise''),
      for a fixed variance $\sigma^2 > 0$.

    \item Draw random, ordered time-points $1 \le t_1 < t_2 < t_3 \le T$.

    \item Set
      \begin{align*}
        x_i
        = 
        \begin{cases}
          \xi_i +  2 & i = t_{\tau(1)} \\
          \xi_i -  3 & i = t_{\tau(2)} \\
          \xi_i + 16 & i = t_{\tau(3)} \\
          \xi_i      & \text{ else}.
        \end{cases}
      \end{align*}
  \end{itemize}
  Here $\tau \in S_3$ is the identity for samples of Class $A$,
  and it is drawn randomly from $S_3 \setminus \{\id_{S_3}\}$ for samples of Class $B$.
  The task is to learn distinguishing the two classes.
  This is again a ``chronological question'', as in the preceeding example.
  We test two different architectures
  \begin{itemize}
    \item FCN: a fully connected network of depth with linear readout. Trainable parameters: $5762$ (for $T=100$). 
    \item ISS: a layer inspired by the current work,
      given by
      \begin{align*}
        \max_{1 \le i_1 < i_2 < i_3 \le T} \left( f^\theta_1(x_{i_1}) + f^\theta_2(x_{i_2}) + f^\theta_3(x_{i_3}) \right),
      \end{align*}
      where $f^\theta_j$ are learnable functions (neural networks themselves, in our implementation),
      followed by a linear ``readout'' layer.
      Trainable parameters: $342$ (for any $T$).
  \end{itemize}

  The results in \Cref{fig:fcniss} show that the FCN (wth a much larger number of parameters),
  does consistently worse than the ISS.
  The inductive bias of ``chronological question'', in this simple example,
  seems to help our architecture, for small noise, to achieve perfect classification.

  \begin{figure}[!ht]
      \centering
      \begin{minipage}{0.45\textwidth}
      \centering
      \includegraphics[width=0.95\textwidth, clip, trim=30 0 40 0]{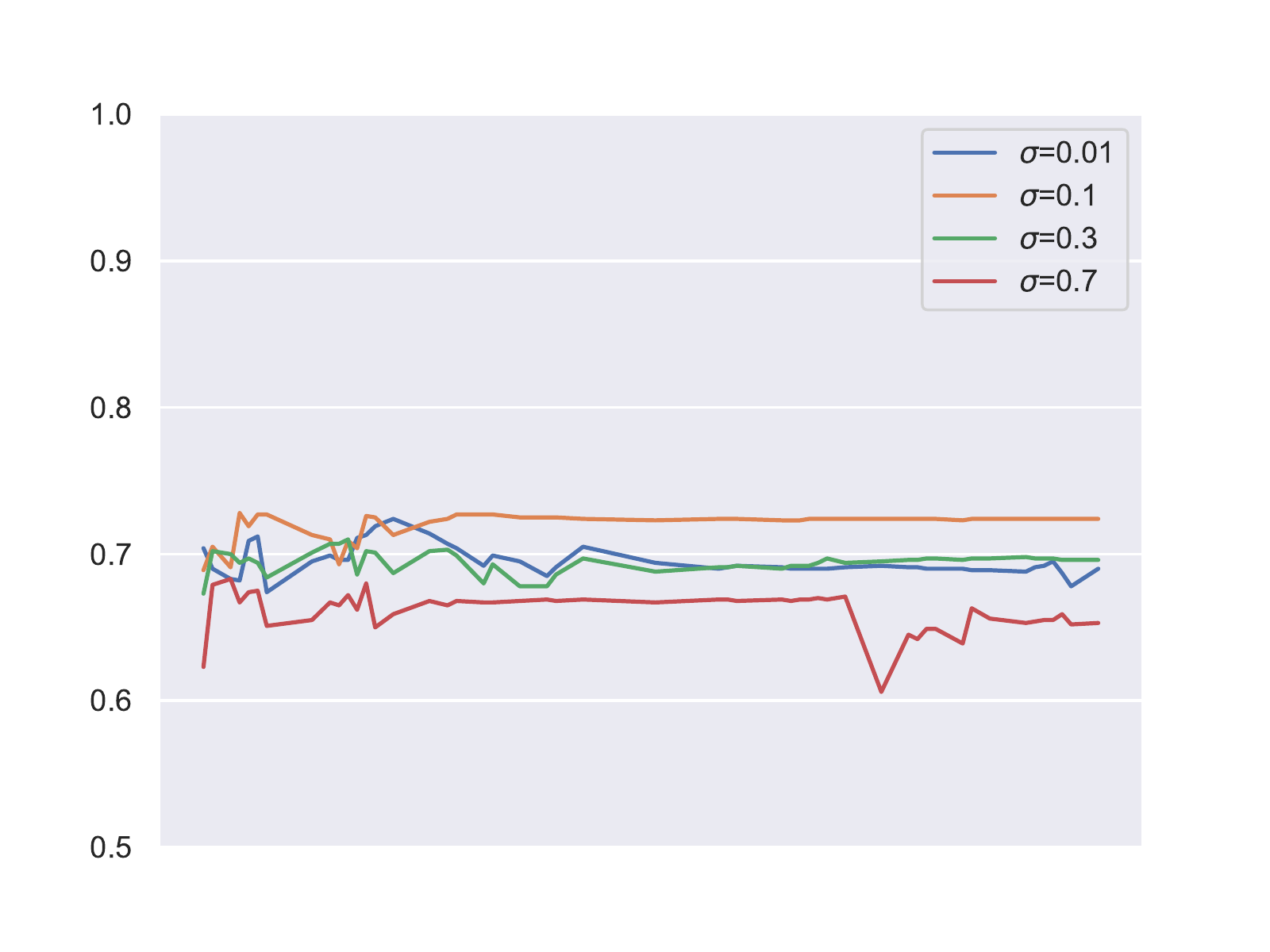}
      \end{minipage}
      \begin{minipage}{0.45\textwidth}
      \centering
      \includegraphics[width=0.95\textwidth, clip, trim=30 0 40 0]{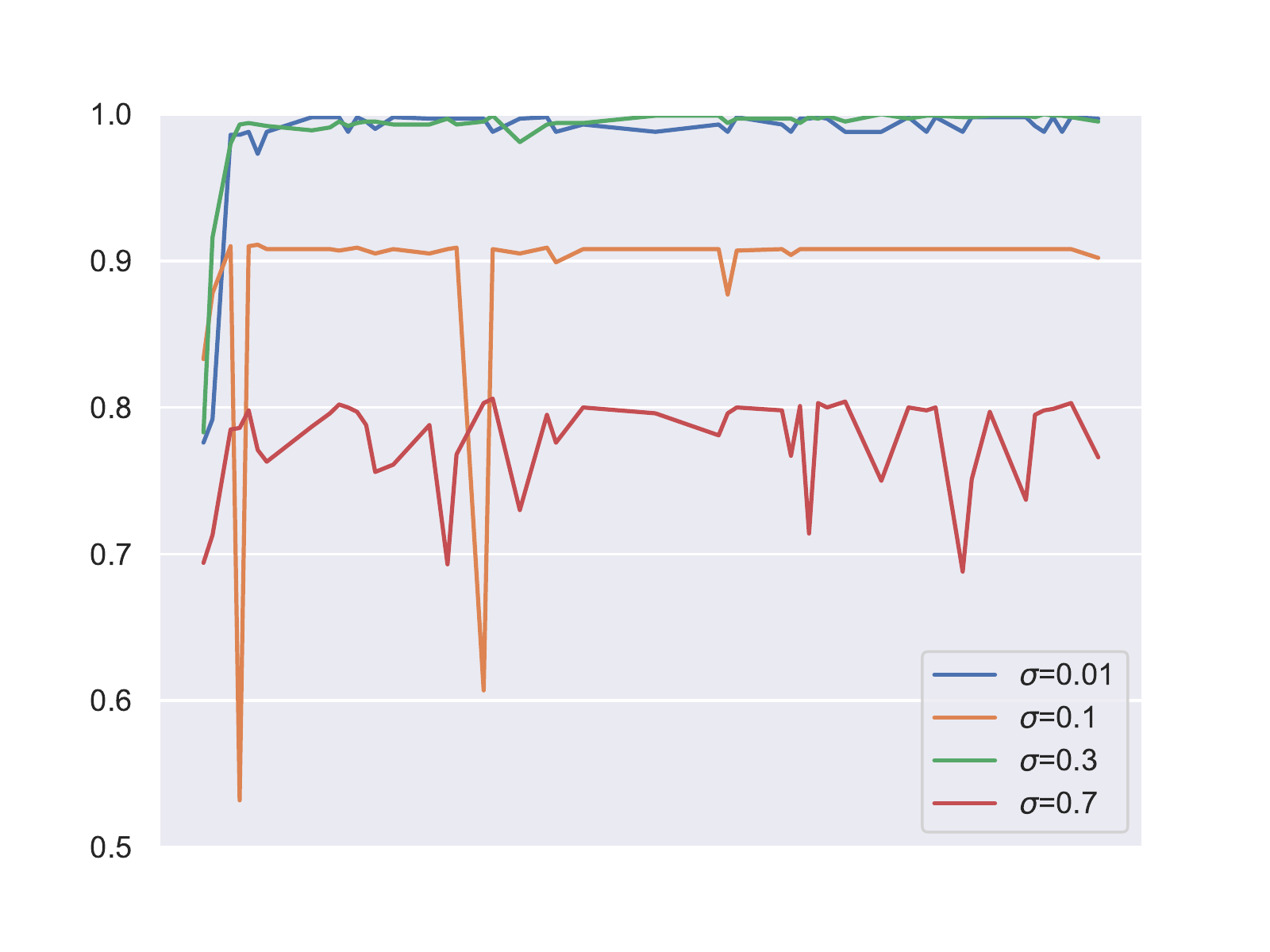}
      \end{minipage}
      \label{fig:fcniss}
      \caption{Epoch test accuracy for the architecture FCN (left) and ISS (right) for different values of $\sigma$.
      ($N=100,$ training size: $1000$, test size: $100$, $100$ epochs).}
  \end{figure}




\end{example}

Abstractly we can describe functions that extract such \emph{chronological features} of sequences
$x \in \R^N$
in the following form
\newcommand\pool{\operatorname{pool}}
\begin{align*}
  x \mapsto \pool\Big( K(x_{i_1},\dots,x_{i_k}) : \{ i_1 < \dots < i_k \} \subset [N] \Big),
\end{align*}
where
\begin{align*}
  K\colon \R^k \to \R, \quad
  \pool\colon \R^{\binom{N}k} \to \R.
\end{align*}
Now, in this generality, such features are computationally intractable, even for modest values of $N$ and $k$,
since $K$ has to be evaluated $\binom{N}k$ times.
The iterated-sums signature presented in this work represents a special case of the functions $K$ and $\pool$ that \emph{is} tractable.
The application of this structure in deep learning pipelines will be addressed in subsequent work.

\end{enumerate}

~\\

The central object of this work, the \emph{iterated-sums signature} $\ISS^\S$
will be properly defined in \eqref{S-ISS} below.
To get there, we need to work through some algebraic background
in \Cref{sec:tss} first.
We therefore now give a preview
in the setting of the tropical semiring $\S=\R_\mp$
with addition $\oplus_\mp$ given by the minimum,
and multiplication $\odot_\mp$ given by the sum of real numbers.
Let $(z_1, z_2, z_3, \dots)$, $z_i \in \R$, be an infinite time series.
Define $\ISS^{\R_\mp}$, indexed by words $w=w_1\cdots w_k$ in the alphabet $A=\{1,2,3,\ldots\}$%
\footnote{
Caution: in the main text we write (in the one-dimensional case) the alphabet as $\{[\w1],[\w1^2],[\w1^3], \dots\} \cong \{1,2,3,\dots\}$, since this notation extends nicely to higher dimensions.}
and $1 \le s \le t < +\infty$, as
\begin{align*}
  \Big\langle \ISS^{\R_\mp}_{s,t}(z), w \Big\rangle
  &\coloneqq
  \sideset{}{_\mp}\bigoplus_{s < j_1 < \dots < j_k\le t} z_{j_1}^{\odot_\mp w_1} \odot_\mp \dots \odot_\mp z_{j_k}^{\odot_\mp w_k} \\
  &=
  \min_{s < j_1 < \dots < j_k \le t} \{ w_1 \cdot z_{j_1} + \dots + w_k\cdot z_{j_k} \}.
\end{align*}
For example (see also \Cref{fig:plots})
\begin{align}
  \label{eq:thirdOrder}
  \begin{split}
  \Big\langle \ISS^{\R_\mp}_{s,t}(z), 1 \Big\rangle   &= \min_{s<j\le t} z_{j}\\
  \Big\langle \ISS^{\R_\mp}_{s,t}(z), 74 \Big\rangle  &= \min_{s<j_1 < j_2\le t} \{ 7\cdot z_{j_1} + 4\cdot z_{j_2} \}\\
  \Big\langle \ISS^{\R_\mp}_{s,t}(z), 714 \Big\rangle &= \min_{s<j_1 < j_2 < j_3\le t} \{ 7\cdot z_{j_1} + z_{j_2} + 4\cdot z_{j_3} \}.
  \end{split}
\end{align}
\begin{figure}[!ht]
    \centering
        \includegraphics[width=0.7\textwidth]{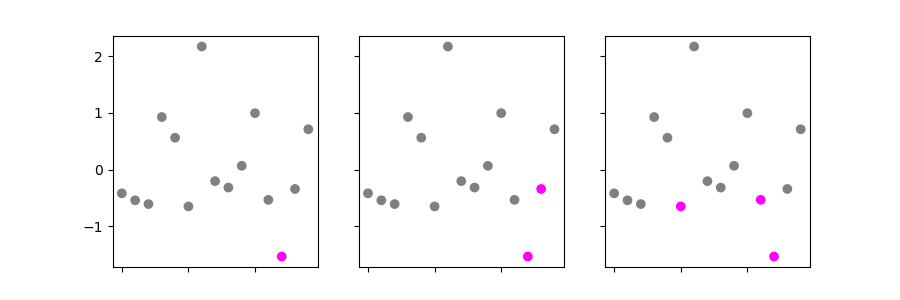}
    \caption{Example of data points (in magenta) attaining the minima in \cref{eq:thirdOrder}.}
    \label{fig:plots}
\end{figure}
We remark two, maybe, non-obvious properties of this object.
Firstly,
in order to calculate $\ISS^{\R_\mp}$ over large intervals, it suffices to
calculate it over small intervals:
\begin{example}
\label{ex:chen}
For $0\le s < t < u$,
\allowdisplaybreaks
\begin{align*}
  &\Big\langle \ISS^{\R_\mp}_{s,u}(z), 74 \Big\rangle  
  =  \sideset{}{_{\mp}}\bigoplus_{s < j_1 < j_2\le u} z_{j_1}^{\odot_{\mp} [\w1^7]} \odot_{\mp} z_{j_2}^{\odot_{\mp} [\w1^4] }\\
  &= \min_{s < j_1 < j_2 \le u } \{ 7 z_{j_1} + 4 z_{j_2} \}\\
  &=
  \min \Big\{
    \min_{ s < j_1 < j_2 \le t}  \{ 7 z_{j_1} + 4 z_{j_2} \},\\
  &\qquad
    \min_{ t < j_1 < j_2 \le u}  \{ 7 z_{j_1} + 4 z_{j_2}\},
    \min_{ s < i  \le t} \{ 7 z_{i}  \}
  +
    \min_{  t < j  \le u} \{ 4 z_{j}\}
  \Big\} \\
  &=
  \Big\langle \ISS^{\R_\mp}_{s,t}(z), 74 \Big\rangle
  \oplus_{{\mp}}
  \Big\langle \ISS^{\R_\mp}_{t,u}(z), 74 \Big\rangle\\
  &\qquad
   \oplus_{{\mp}} \left( 
     \Big\langle \ISS^{\R_\mp}_{s,t}(z), 7 \Big\rangle
    \odot_{{\mp}}
  \Big\langle \ISS^{\R_\mp}_{t,u}(z), 4 \Big\rangle \right).
\end{align*}
\end{example}
This is also the reason why expressions that seem to have polynomial complexity (after all, the third-order iterated sum in \eqref{eq:thirdOrder} takes the maximum of
$\O(|t-s|^3)$-terms), are in fact calculable in linear time, as we will see in \Cref{sec:algo}.

Secondly, we note that the \emph{product} (in the semiring) of any 
iterated sums (which are iterated minima in the current example)
can be written as the \emph{sum} (in the semiring) of (different) iterated sums:
\begin{example}
\label{ex:qs}
In the min-plus semiring $\S=\R_\mp$ we have
\begin{align*}
  &\Big\langle \ISS^{\R_\mp}_{s,t}(z), 1 \Big\rangle
  \odot_\s
  \Big\langle \ISS^{\R_\mp}_{s,t}(z), 74 \Big\rangle \\
  &=
  \min_{s<i\le t} \{ z_i \}
  +
  \min_{s<j < k\le t} \{ 7 z_j  + 4 z_k \} \\
  &=
  \min_{\substack{s<i\le t \\s<j<k\le t}} \{ z_i + 7 z_j  + 4 z_k \} \\
  &=
  \min\Big\{
  \min_{s<i<j<k\le t} \{ z_i + 7 z_j  + 4 z_k \},
  \min_{s<i=j<k\le t} \{ z_i + 7 z_j  + 4 z_k \},
  \min_{s<j<i<k\le t} \{ z_i + 7 z_j  + 4 z_k \}, \\
  &\qquad\qquad
  \min_{s<j<k=i\le t} \{ z_i + 7 z_j  + 4 z_k \},
  \min_{s<j<k<i\le t} \{ z_i + 7 z_j  + 4 z_k \} \Big\} \\
  &=
  \Big\langle \ISS^{\R_\mp}_{s,t}(z), 174 \Big\rangle
  \oplus_\s
  \Big\langle \ISS^{\R_\mp}_{s,t}(z), 84 \Big\rangle
  \oplus_\s
  \Big\langle \ISS^{\R_\mp}_{s,t}(z), 714 \Big\rangle\\
  &\qquad\qquad
  \oplus_\s
  \Big\langle \ISS^{\R_\mp}_{s,t}(z), 75 \Big\rangle
  \oplus_\s
  \Big\langle \ISS^{\R_\mp}_{s,t}(z), 741 \Big\rangle.
\end{align*}
\end{example}

Both of these facts might come as no surprise to people familiar
with iterated integrals or iterated sums over fields.
Indeed,
the first property is a version of \emph{Chen's identity}.
It just says that the computation of iterated integrals and iterated sums
can be split into calculations on subintervals.
This property is usually encoded algebraically
by the non-cocommutative \emph{deconcatenation} coproduct on the unital tensor algebra over an alphabet.
The general form of Chen's identity in our setting is stated in \Cref{lem:chen}.

Integration by parts implies that linear combinations of iterated integrals are closed under multiplication.
This finds its abstract algebraic formulation in terms of the commutative shuffle product on the unital tensor algebra over an alphabet \cite{Che1957,Ree1958}.
Analogously, its discrete counterpart, i.e., summation by parts,
permits to define an algebra on iterated sums, leading to the notion of commutative quasi-shuffle algebra \cite{Gai1994}.
The general form of the quasi-shuffle identity in our setting is found in \Cref{lem:qsIdentity}.

Maybe more interestingly, new phenomena appear when working over general semirings.
As we will see in \Cref{sec:timeWarpingInvariants},
over an \emph{idempotent} semiring,
non-strict iterated sums satisfy a \emph{shuffle} identity and in this sense behave like iterated \emph{integrals}.
These non-strict iterated sums also give
a nice way to get certain time-warping invariants of a real-valued time series,
covering expression \eqref{eq:min}.

\smallskip 

The paper is organized as follows. 
In \Cref{sec:tss} we define the iterated-sums signature over a commutative semiring.
In \Cref{sec:qsym} we take a closer look at quasisymmetric functions, which are underlying the
iterated-sums signature and are of independent interest.
In \Cref{sec:timeWarpingInvariants} we return to the question of time warping invariants in the context of iterated-sums signature over a commutative semiring.
\Cref{sec:algo} exposes an algorithm for efficient calculation of iterated sums. It is based on Chen's identity,
but alternatively can be considered as a dynamic programming principle in play here.
We finish with conclusions and an outlook in Section \ref{sec:concl}.
In Appendix \ref{sec:category} we present a categorical view on semirings, semimodules, and semialgebras.
Such a categorical view is useful in highlighting the similarities
to the theory of rings, modules, and algebras.

\textbf{Related works}

We finish this introduction by mentioning related literature.
We already indicated how
iterated sums and integrals have, in the last decade, been successfully applied
as a feature extraction method in machine learning.
The relation of iterated sums to the Hopf algebra of quasisymmetric functions \cite{MR1995} was established in \cite{DET2020}.

Symmetric functions form an important subspace and the generalization of
this subspace to the setting of semirings has been investigated in \cite{CK2016,KL2020,KL2019}.

Semirings play an important role in computer science. They appear, for example, in the closely related fields of 
language processing \cite{Goo1999}, 
the theory of algorithms \cite{cygan2015parameterized,gondran2008graphs}, 
the theory of weighted automata \cite{baccelli1992synchronization,sakarovitch2009elements},
of Petri nets \cite{heidergott2014max},
shortest-paths problems in weighted directed graphs \cite{Moh2002,Fle1980},
and iteration theories \cite{bloom1993iteration}.

The tropical semiring, in particular, has been intensely studied, for example
in algebraic geometry \cite{MS2015},
in statistics \cite{pachter2004tropical},
in economics \cite{baldwin2013tropical},
and in biology \cite[Section 2]{pachter2005algebraic}.
Its linear algebra is well-understood \cite{butkovivc2010max,akian2006max}.

\medskip

\section{Iterated-sums signatures over a semiring}
\label{sec:tss}

We start by introducing basic concepts from semiring theory. Relevant references are \cite{KS1986,MS2015,Wor2009}. The definitions and constructions recalled here are in some sense ``hands-on''. For a more abstract, i.e. categorical, view see \Cref{sec:category}. 

\begin{definition}
	A \textbf{monoid} is a non-empty set \(M\) together with an associative product \(\cdot\colon M\times M\to M\) and a neutral element \(\one_M\in M\), i.e., \(m \cdot\one_M=\one_M\cdot m=m\) for all \(m\in M\).
\end{definition}

For example, starting from an alphabet $A=\{a_1,\ldots,a_n\}$, the set of (finite) words over $A$ forms, under concatenation, the free non-commutative monoid, denoted by $A^*$.
The empty word, $\emptyWord$, is the neutral element. The length of a word $w=a_{i_1} \cdots a_{i_k} \in A^*$ is denoted $\len(w)=k$. A \textbf{monoid morphism} is a map between monoids which is compatible with the products
and takes the neutral element to the neutral element.
Note that for an alphabet $A$ and any monoid $(M,\centerdot,1_M)$, every map $\phi\colon A \to M$ can be uniquely extended to a monoid morphism  $\hat{\phi}\colon A^* \to M$ by defining $\hat{\phi}(\emptyWord)=1_M$ and for any word $w=a_{i_1} \cdots a_{i_k} \in A^*$, $\hat{\phi}(a_{i_1} \cdots a_{i_k})\coloneqq\phi(a_{i_1}) \centerdot \cdots \centerdot \phi(a_{i_k})$. In other words, $A^*$ is the \textbf{free monoid} over $A$.

The notion of semiring combines two monoids in a compatible, i.e.~distributive, way.
However, contrary to rings, the invertibility under addition is not part of the requirements.  

\begin{definition}
\label{def:semiring}
  The tuple $(\mathbb{S},\oplus_\s,\odot_\s,\zero_\s,\one_\s)$ is a \textbf{commutative semiring} if
  \begin{itemize}
		\item $(\mathbb{S},\oplus_\s,\zero_\s)$ is a commutative monoid  with unit $\zero_\s$,
		\item $(\mathbb{S},\odot_\s,\one_\s)$ is a commutative monoid with unit $\one_\s$, 
		\item $\zero_\s \odot_\s s = s \odot_\s \zero_\s = \zero_\s$ for all $s \in \mathbb{S}$, and
    \item multiplication distributes over addition
      \begin{align}
        \label{eq:distributes}
      a \odot_\s (b\oplus_\s c) 
      = (a\odot_\s b) \oplus_\s (a \odot_\s c),
      \quad 
      (a \oplus_\s b) \odot_\s c = (a\odot_\s c) \oplus_\s (b \odot_\s c).
      \end{align}
  \end{itemize}
\end{definition}

Note that the parentheses on the right-hand sides of the identities in \eqref{eq:distributes} can be omitted when assuming the common precedence of multiplication over addition.
A semiring $\mathbb{S}$ is called \textbf{idempotent}, if for all elements $a \in \mathbb{S}$ we have that $a \oplus_\s a = a$.

Semirings form an essential part of the modern theory of automata and languages. We refer the reader
to the introductory references \cite{KS1986,CL2008}.

\begin{example}~
  Any commutative ring, in particular the field of reals $(\R,+,\cdot,0,1)$,
  forms a commutative semiring.

The paradigms of ``honest'' semirings whose underlying sets are subsets of the reals are  
\begin{enumerate}
  \item $\R_{\mp}{\mathrm{:}}$ min-plus semiring $(\R \cup \{+\infty\},\min,+,+\infty,0)$
  \item $\R_{\maxp}{\mathrm{:}}$  max-plus semiring $(\R \cup \{-\infty\},\max,+,-\infty,0)$,
\end{enumerate}
which are also known as tropical, respectively arctic, semirings.
Here minimum, respectively maximum, are considered as binary operations replacing the usual additive structure on $\R$, and addition becomes multiplication. This results in particular arithmetic rules, e.g., $3 \oplus_{\scriptscriptstyle{{\maxp}}} 3 = 3$, $4 \oplus_{\maxp} 3 = 4$, and $3 \odot_{\maxp} 3 = 6$, $-1 \odot_{\maxp} -1 = -2$.    
\label{ex:semirings}
  \begin{enumerate}
  \setcounter{enumi}{2}
  \item
		possibilistic semiring\footnote{Also known as Viterbi or Bayesian semiring \cite{SW2018}.} $([0,1],\max,\cdot,0,1)$,
	\item $\N{\mathrm{:}}$ non-negative integers $(\N,+,\cdot,0,1)$,
	\item $\R_{\scriptscriptstyle{\mathrm{max-min}}}{\mathrm{:}}$ bottleneck semiring $(\R \cup \{\pm \infty\},\max,\min,-\infty,+\infty)$,
	\item $\overline{\R}_{\scriptscriptstyle{\mathrm{max}}}{\mathrm{:}}$ completed max-plus semiring $(\R \cup \{\pm\infty\},\max,+,-\infty,0)$,
	\item expectation semiring (or gradient semiring) \cite{Eis2002}, $(\R_{\ge 0} \times V, \oplus, \odot, (0,0),(1,0))$ with $V$ an arbitrary $\R$-vector space, and
  \allowdisplaybreaks
    \begin{align*}
      (a,v) \oplus (a',v') &\coloneqq (a+a', v+v') \\
      (a,v) \odot (a',v')  &\coloneqq (a \cdot a', a' v + a v').
    \end{align*}
  \end{enumerate}

  There are also examples of semirings whose underlying sets are \emph{not} given by (subsets of) the real line.
  \begin{enumerate}
    \setcounter{enumi}{7}
		\item semiring of (bounded) polytopes \cite[Proposition 2.23]{pachter2005algebraic},
		\item $k$-best proof semiring \cite{Goo1999,GS2009},
		\item $k$-tropical semiring \cite{Moh2002},
		\item semiring of formal languages \cite{DK2009},
		\item semiring of binary relations \cite{DK2009},
		\item semiring of subsets of a set $M$ \cite[Beispiel 2.10.a)]{NS2018} $(2^M, \cup, \cap, \emptyset, M)$, and
    \item $\mathbb{B} {\mathrm{:}}$ Boolean semiring $(\{\text{false},\text{true}\}, \mathrm{or}, \mathrm{and}, \text{false}, \text{true})$.
  \end{enumerate}
  Regarding the last two examples, in fact, any distributive lattice (with minimal and maximal element) naturally yields a commutative semiring,
  \cite[Proposition 2.25]{golan2013semirings}.

  \begin{enumerate}
    \setcounter{enumi}{14}
  \item semirings constructed from t-norms \cite[Example 6]{KW2008}.
  \end{enumerate}
\end{example}

\smallskip

We note that summation by parts holds in a semiring. Indeed, multiplying two sums%
\footnote{
  In order to not clutter the notation, for the moment we do not
specify bounds for the summation indices.
To have a well-defined sum it suffices to assume that the sequences $a_1,a_2,\dotsc$ and $b_1,b_2,\dotsc$ are eventually equal to $\zero_\s$.
}
over semiring elements $a_i,b_j \in  \mathbb{S}$ yields
  \allowdisplaybreaks
   \begin{align}
 \label{sumbyparts}
\begin{split} 
   \lefteqn{\left( \sideset{}{_\s}\bigoplus_{i} a_i \right)
    \odot_\s
    \left( \sideset{}{_\s}\bigoplus_{j} b_j \right)
    =
    \sideset{}{_\s}\bigoplus_{i,j} a_i \odot_\s b_j} \\
    &=
    \sideset{}{_\s}\bigoplus_{i<j} \left( a_i \odot_\s b_j \right)
    \oplus_\s 
    \sideset{}{_\s}\bigoplus_{j<i} \left( a_i \odot_\s b_j \right)
    \oplus_\s 
    \sideset{}{_\s}\bigoplus_{i} \left( a_i\odot_\s b_i \right).
  \end{split}
\end{align}   
Regarding iterated sums of depth\footnote{Here, \emph{depth} of an iterated sum refers to the number of indices in the summation.} $k$
\begin{align}
\label{iteratedsums} 
	\sideset{}{_\s}\bigoplus_{j_k}
    \sideset{}{_\s}\bigoplus_{j_{k-1} < j_k}
    \dots
    \sideset{}{_\s}\bigoplus_{j_{1} < j_2}
      a_{j_1}\odot_\s \dots \odot_\s a_{j_k}
   =
	\sideset{}{_\s}\bigoplus_{j_1 < \dots 
 	< j_k} a_{j_1}\odot_\s \dots \odot_\s a_{j_k},
\end{align}
one can show that, like in the usual ring-setting, \eqref{sumbyparts} implies that the product of two iterated sums of depths $k_1$ and $k_2$ can be expressed in terms of a linear combination of iterated sums of depths $k$, with $\max(k_1,k_2) \leq k \leq k_1+k_2$. The reduction of the lengths of the iterated sums is due to the diagonal terms (, e.g., the third summand on the right-hand side of \eqref{sumbyparts}). The algebraic formulation of this leads to the quasi-shuffle identity (see below). 

\smallskip

Regarding \eqref{sumbyparts}, we remark on a peculiarity in the semiring setting.
The lack of inverses with respect to addition makes
summation-by-parts for non-strict iterated sums less appealing
since the space of such sums is not closed under multiplication. 
Indeed, replacing in \eqref{iteratedsums}
the strict inequalities of the indices by \emph{non-strict}
inequalities and multiplying two such expressions
one arrives at the sum of expressions
that \emph{do} contain strict inequalities as well.
For example
\begin{align}
 \label{nonstrictsums}
\begin{split} 
 &\left( \sideset{}{_\s}\bigoplus_{i_1 \le i_2} z_{i_1}^{\odot_\s 7} \odot_\s z_{i_2}^{\odot_\s 3} \right)
    \odot_s
    \left( \sideset{}{_\s}\bigoplus_{i_3} z_{i_3}^{\odot_\s 5} \right) \\
    &=
    \left( \sideset{}{_\s}\bigoplus_{i_1 \le i_2 \le i_3} z_{i_1}^{\odot_\s 7} \odot_\s z_{i_2}^{\odot_\s 3} \odot_\s z_{i_3}^{\odot_\s 5} \right)
    \oplus_\s
    \left( \sideset{}{_\s}\bigoplus_{i_1 \le i_3 < i_2} z_{i_1}^{\odot_\s 7} \odot_\s z_{i_2}^{\odot_\s 3} \odot_\s z_{i_3}^{\odot_\s 5} \right) \\
    &\qquad
    \oplus_\s
    \left( \sideset{}{_\s}\bigoplus_{i_3 < i_1 \le i_2} z_{i_1}^{\odot_\s 7} \odot_\s z_{i_2}^{\odot_\s 3} \odot_\s z_{i_3}^{\odot_\s 5} \right). \\
\end{split}
\end{align}
However, if we consider \eqref{nonstrictsums} in an \emph{idempotent} semiring,
one observes an interesting phenomenon.
The fact that $a \oplus_\s a = a$ for all elements $a$ in such a semiring implies the somewhat surprising identity 
\begin{align}
 \label{Idemsumbyparts}
\left( \sideset{}{_\s}\bigoplus_{i} a_i \right)
    \odot_\s
    \left( \sideset{}{_\s}\bigoplus_{j} b_j \right)
    =
    \sideset{}{_\s}\bigoplus_{i \le j} \left( a_i \odot_\s b_j \right)
    \oplus_\s 
    \sideset{}{_\s}\bigoplus_{j\le i} \left( a_i \odot_\s b_j \right),
\end{align}
where we used that 
$$ 
	\sideset{}{_\s}\bigoplus_{i} \left( a_i\odot_\s b_i \right) \oplus_\s 
    \sideset{}{_\s}\bigoplus_{i} \left( a_i\odot_\s b_i \right) = \sideset{}{_\s}\bigoplus_{i} \left( a_i\odot_\s b_i \right),
$$
The example
\eqref{nonstrictsums} then becomes
 \allowdisplaybreaks{
\begin{align}
 \label{shufflesum}
 \begin{split}
 &\left( \sideset{}{_\s}\bigoplus_{i_1 \le i_2} z_{i_1}^{\odot_\s 7} \odot_\s z_{i_2}^{\odot_\s 3} \right)
    \odot_s
    \left( \sideset{}{_\s}\bigoplus_{i_3} z_{i_3}^{\odot_\s 5} \right) \\
    &=
    \left( \sideset{}{_\s}\bigoplus_{i_1 \le i_2 \le i_3} z_{i_1}^{\odot_\s 7} \odot_\s z_{i_2}^{\odot_\s 3} \odot_\s z_{i_3}^{\odot_\s 5} \right)
    \oplus_\s
    \left( \sideset{}{_\s}\bigoplus_{i_1 \le i_3 \le i_2} z_{i_1}^{\odot_\s 7} \odot_\s z_{i_2}^{\odot_\s 3} \odot_\s z_{i_3}^{\odot_\s 5} \right) \\
    &\qquad
    \oplus_\s
    \left( \sideset{}{_\s}\bigoplus_{i_3 \le i_1 \le i_2} z_{i_1}^{\odot_\s 7} \odot_\s z_{i_2}^{\odot_\s 3} \odot_\s z_{i_3}^{\odot_\s 5} \right).
    \end{split}
\end{align} }    

Hence, in the idempotent case, products of non-strict iterated sums satisfy the \emph{shuffle} relation.
We return to this in \Cref{sec:timeWarpingInvariants}.

\begin{definition}
\label{def:semimod}
  A \textbf{semimodule} over a commutative semiring $(\mathbb{S},\oplus_\s,\odot_\s,\zero_\s,\one_\s)$ consists of a commutative monoid $(M,+_M,\zero_M)$ and a scalar multiplication $\mathbb{S} \times M \to M$, $(s,m) \mapsto sm$, satisfying for all $s,s' \in \mathbb{S}$ and $m,m' \in M$ 
  \begin{align*}
    \one_\s m = m 			&\qquad \zero_\s m = \zero_M \\
    (s \odot_\s s') m  = s (s' m) 	&\qquad (s \oplus_\s s') m = s m +_M s' m  \\
    s (m +_M m') &= s m +_M s m'.
  \end{align*}  
Note that if the underlying semiring $\mathbb{S}$ is idempotent, then the semimodule $M$ is idempotent as well.

Let $(N,+_N,\zero_N)$ be another semimodule over $\mathbb{S}$. A map $\phi\colon M \to N$ is a \textbf{semimodule morphism} (or an \textbf{$\S$-linear map}) if for all $s,s' \in \mathbb{S}$ and $m,m' \in M$
  \begin{align*}
    \phi(sm +_M s'm') &= s\phi(m) +_N s'\phi(m').
  \end{align*}
\end{definition}

\begin{example}
An incarnation of the  \textbf{free $\mathbb{S}$-semimodule} $\mathbb{F}$ on a set $D$  is given by functions $f\colon D \to \mathbb{S}$ with finite support, i.e., $f(d) = \zero_\s$ for all but finitely many elements $d \in D$. The action of $\mathbb{S}$ as well as the addition are defined pointwise.
\end{example}

\begin{definition}
  \label{def:semialg}
  An associative \textbf{semialgebra} over a commutative semiring $\mathbb{S}$ consists of a semiring $(\mathbb{A},\oplus_\y,\odot_\y,\zero_\y,\one_\y)$ such that $(\mathbb{A}, \oplus_\y,\zero_\y)$ is a semimodule over $\mathbb{S}$ and such that the semimodule structure is compatible with $\odot_\y$ in the following way
  \begin{align*}
    s (a \odot_\y a') = (s a) \odot_\y a' = a \odot_\y (s a').
  \end{align*}
\end{definition}

\begin{remark}
Motivated by summation by parts \eqref{sumbyparts} and the particular property of iterated non-strict sums \eqref{Idemsumbyparts}, one can introduce the notion of a Rota--Baxter $\mathbb{S}$-semialgebra.
  Let $\mathbb{A}$ be a  $\mathbb{S}$-semialgebra. A \emph{Rota--Baxter map} of weight $\lambda \in \mathbb{S}$ is a $\mathbb{S}$-linear map $R\colon\mathbb{A} \to \mathbb{A}$ satisfying for any $x,y \in \mathbb{A}$
\begin{equation}
\label{semiRBlambda}
	R(x) \odot_\y R(y) 
	= R\big(R(x)\odot_\y y 
	\oplus_\y x\odot_\y R(y) \big)
	\oplus_\y \lambda R(x \odot_\y y).
\end{equation}
Note that if the semiring $\mathbb{S}$, and therefore also the semialgebra $\mathbb{A}$, is idempotent the map $\tilde{R}\coloneqq\lambda \id_\y \oplus_\y R$ satisfies the particular relation \eqref{semiRBlambda} as well.
In fact, we have the more surprising (weight zero) identity%
\footnote{
Indeed, we see that by expanding the right-hand side of \eqref{semiRBzero} we obtain
  \begin{align*}
  \lefteqn{	\tilde{R}\big(\tilde{R}(x)\odot_\y y 
	\oplus_\y x \odot_\y \tilde{R}(y) \big) 
	= \tilde{R}\big( (\lambda \oplus_\s  \lambda) x \odot_\y y 
		\oplus_\y   R(x) \odot_\y y 
		\oplus_\y  x \odot_\y R(y) \big)}\\
	&= (\lambda \odot_\s  \lambda) x \odot_\y y 
		\oplus_\y    \lambda R(x) \odot_\y y 
		\oplus_\y  \lambda x \odot_\y R(y)  	 
		   \oplus_\y  R\big(
	\lambda x \odot_\y y 
	\oplus_\y R(x)\odot_\y y 
	\oplus_\y x\odot_\y R(y) \big)\\
	&=  \lambda \odot_\s \lambda (x \odot_\y y) 
	\oplus_\y  \lambda x \odot_\y R(y) 
	\oplus_\y  \lambda R(x) \odot_\y y  
	\oplus_\y  R(x) \odot_\y R(y)
	=\tilde{R}(x) \odot_\y \tilde{R}(y) .
\end{align*}}
(compare also \eqref{shufflesum})
\begin{equation}
\label{semiRBzero}
	\tilde{R}(x) \odot_\y \tilde{R}(y) 
	= \tilde{R}\big(\tilde{R}(x)\odot_\y y 
	\oplus_\y x\odot_\y \tilde{R}(y) \big).
\end{equation}
\end{remark}

\smallskip

We now consider formal series over the (possibly infinite) alphabet $A$ with coefficients in a commutative semiring $\mathbb{S}$ of the form
\begin{equation} 
  \label{formal}
	F \coloneqq \sum_{w \in A^*} c_w w,\quad c_w \in \S.
\end{equation}
The set of all such series is denoted by $\mathbb{S}\llangle A \rrangle$. For a series $F \in \mathbb{S}\llangle A\rrangle$, the support, $\mathrm{supp}(F)$, consists of all words in $w\in A^*$ with coefficient $c_w$ different from $\zero_\s$. We denote with $\mathbb{S}\langle A\rangle$ the subset of series with finite support. We may view a series \eqref{formal} as a map $F\colon A^* \to \S$, with
\begin{align*}
  \langle F, w \rangle_\s \coloneqq F(w) \coloneqq c_w.
\end{align*}
Each map $F\colon A^*\to\mathbb S$ can be uniquely extended to a linear map $F\colon\mathbb{S}\langle A\rangle\to \mathbb{S}$,
and we denote this pairing by $\langle ., . \rangle_\s\colon\mathbb S\llangle A\rrangle\times A^*\to\mathbb S$ still.

We can equip $\mathbb{S}\llangle A\rrangle$ with a linear and multiplicative structure by defining
\begin{align*}
  \langle sF,w \rangle_\s &\coloneqq s\langle F,w \rangle_\s, \quad s \in \S \\
	\langle F_1 + F_2,w \rangle_\s 
	  &\coloneqq \langle F_1,w \rangle_\s \oplus_\s \langle F_2,w \rangle_\s \\
	\langle F_1F_2,w \rangle_\s 
	  &\coloneqq \sideset{}{_\s}\bigoplus_{vu=w}\langle F_1,v \rangle_\s\odot_\s\langle  F_2,u \rangle_\s.
\end{align*}

For instance, let $F= s_1 a_1 a_2  + s_2 a_1$ and $G = t_1 a_3 + t_2 a_2a_3$ be elements in $\mathbb{S}\langle A\rangle$ then 
 \begin{align*}
    FG
      &=
      \left( (s_1 \odot_\s t_1) \oplus_\s (s_2 \odot_\s t_2) \right) a_1 a_2 a_3
      +
      \left( s_1 \odot_\s t_2 \right) a_1a_2a_2a_3
      +
      \left(s_2 \odot_\s t_1\right) a_1a_3.
  \end{align*}
This turns $\mathbb{S}\llangle A\rrangle$, as well as $\mathbb{S}\langle A\rangle$, into $\mathbb{S}$-semialgebras.  
The constant series are given by defining for any element $s \in \mathbb{S}$ the series $F_s\coloneqq s \emptyWord$, which include in particular the constant series $1$ and $0$.\footnote{Recall that $\emptyWord\in A^*$ is the empty word.}

As is the case over rings, $\mathbb{S}\langle A\rangle$ is the \textbf{free associative $\S$-semialgebra} over the alphabet $A$. This manifests in the following universal property: 
\begin{proposition}
	\label{prp:SA.univ}
	For any $\mathbb{S}$-semialgebra $\mathbb{U}$ and map $\phi\colon A \to \mathbb{U}$ there exists a unique $\mathbb{S}$-semialgebra morphism extending $\hat\phi \colon \S\langle A \rangle \to \mathbb{U}$ extending $\phi$.
	In other words, every map \(\phi\colon A\to\mathbb U\) admits a unique mutiplicative and \(\mathbb S\)-linear extension to \(\mathbb S\langle A\rangle\).
\end{proposition}
The proof of \Cref{prp:SA.univ} is standard, and therefore we omit it. For an indirect proof, we refer the reader to \Cref{sec:category}.

We shall now equip the $\mathbb{S}$-semimodule $\mathbb{S}\langle A\rangle$ with another product.
This product is a natural extension of the well-known shuffle (or Hurwitz) product commonly defined
in automata theory \cite{KS1986}. For this we assume that the alphabet $A$ carries a commutative semigroup product denoted by the binary bracket operation $[- \, -]\colon A \times A \to A$. Observe that commutativity and associativity permit to express iterations $[a_{i_1}[ \cdots [a_{i_{n-1}}\, a_{i_n} ] \cdots ] \eqqcolon [a_{i_1} \cdots a_{i_n}]$. The commutative \textbf{quasi-shuffle product} on $\mathbb{S}\langle A\rangle$ is defined first on words and then extended bilinearly. For words $u a_i$ and $v a_j$, where $a_i,a_j \in A$, $u,v \in A^*$, we define $ua_i \qs \emptyWord = ua_i =\emptyWord \qs ua_i$ and inductively
\begin{align}
\label{Squasishuff}
\begin{split} 
  ua_i \qs va_j &\coloneqq (u \qs va_j)a_i + (ua_i \qs v)a_j + (u \qs v) [a_i\ a_j]. 
\end{split}
\end{align} 
For instance, $a_i \qs a_j = a_ia_j + a_ja_i + [a_i\ a_j]$. It is easy to observe that for a trivial bracket product on $A$, the quasi-shuffle product \eqref{Squasishuff} reduces to the \emph{shuffle product} on $A^*$. The latter will be denoted with $\shuffle$ and satisfies the recursion \cite{Reu2003}
 \begin{align}
\label{Sshuff}
  ua_{i} \shuffle va_{j} &= (u \shuffle va_j) a_i + (ua_i \shuffle v) a_j.
\end{align} 
It is known that an explicit expression can be defined for the shuffle product in terms of so-called shuffle permutations (bijections). In the case of the quasi-shuffle product \eqref{Squasishuff}, an analogous non-recursive formula can be given in terms of certain surjections \cite{EFM2017,EFMPW15}. We refer the reader to \Cref{sec:quasiShuffle} for details.

\textbf{For the remainder of the paper, we specialize to a specific alphabet $A$.}
Let ${A'}=\{1,2,\ldots,d\}$ and let $A$ be the -- extended -- alphabet containing all formal brackets in elements of $A'$, i.e., all formal monomials in those letters,
\begin{align}
\label{theAlphabet}
  A = \{ [\w1],[\w2],\ldots,[\w{d}], [\w1^2], [\w1\w2], \ldots, [\w{d}^2], [\w1^3], \ldots \}.
\end{align}
Here, for consistency of notation, we write $[\w1] = \w1, \ldots, [\w{d}]=\w{d}$.

We consider the space \(\mathbb S^{d,\N_{\ge1}}_{\zero_\s}\) of $\mathbb{S}^d$-valued time
series of infinite length that are eventually equal to $\zero^d_\s$, whose elements are sequences $z=(z_1,z_2,\ldots, z_N, \zero^d_\s, \zero^d_\s, \ldots)$ of elements $z_i=(z^{(1)}_i,\ldots,z^{(d)}_i) \in \mathbb{S}^d$.  

\begin{example}
  \label{ex:logTimeSeries}
  Let $x=(x_0,x_1,x_2,\ldots)$, $x_n \in \R^d$, be a time series that is eventually constant,
  then $z$ with entries
  \begin{align*}
    z_n^{(i)} \coloneqq -\log |x_n^{(i)}-x_{n-1}^{(i)}|, \quad n = 1,2,3, \ldots
  \end{align*}
  is in $\mathbb{S}_{\zero_\s}^{d,\N_{\ge 1}}$, for $\mathbb{S}$ the tropical semiring  $\R_\mp = (\R \cup \{+\infty\},\min,+,+\infty,0)$.
\end{example}

We now define \textbf{$\mathbb{S}$-iterated-sums signature}
$\ISS_{s,t}^\S(z) \in \S\llangle A \rrangle$ for
$z \in \mathbb{S}_{\zero_\s}^{d,\N_{\ge 1}}$,
by defining the coefficient
\begin{align}
\label{S-ISS}  
  \left\langle \ISS_{s,t}^\S(z), w\right\rangle_{\!\s}
  &\coloneqq \sideset{}{_\s}\bigoplus_{s < j_1 < \dots < j_k< t+1} z_{j_1}^{\odot_\s w_1} \odot_\s \dots \odot_\s z_{j_k}^{\odot_\s w_k}
  ,\quad 0 \le s \le t \le +\infty.
\end{align}
for $w=w_1 \cdots w_k \in A^*$, $w_i \in A$, where
for a letter $w_i =[a_{i_1} \cdots a_{i_m}] \in A$ 
$$
    z_{j}^{\odot_\s w_i}
    \coloneqq
	z_{j}^{\odot_\s [a_{i_1} \cdots a_{i_m}] }
	\coloneqq z^{(a_{i_1})}_{j}\odot_\s \cdots \odot_\s  z^{(a_{i_m})}_{j}. 
$$
We also write $\ISS^\S(z) \coloneqq \ISS_{0,\infty}^\S(z)$.
As an example, we compute
\begin{align*}
	 \left\langle \ISS_{s,t}^\S(z), [\w1][\w2\w3]\right\rangle_{\!\s}
	 &= \sideset{}{_\s}\bigoplus_{s < j_1 < j_2< t+1} z_{j_1}^{\odot_\s [\w1] } \odot_\s z_{j_2}^{\odot_\s [\w2\w3]}
	 =  \sideset{}{_\s}\bigoplus_{s < j_1 < j_2< t+1} z_{j_1}^{(1)} \odot_\s z_{j_2}^{(2)} \odot_\s z_{j_2}^{(3)}.
\end{align*}

Our first results concern the verification that $\ISS^\S$ is a proper iterated-sums signature. By this we mean that it carries the two main properties mentioned in the introduction, i.e., it satisfies Chen's identity and is compatible with the quasi-shuffle product \eqref{Squasishuff}.

\begin{lemma}[Chen's identity]
  \label{lem:chen}
  For $p < r < q$, $w \in \mathbb{S}\langle A\rangle$ and $z \in \mathbb{S}_{\zero_\s}^{d,\N_{\ge 1}}$
  \begin{align}
  \label{eq:chen}
    \left\langle \mathrm{ISS}_{p,q}^\mathbb{S}(z), w \right\rangle_{\!\s}
    =
    \sideset{}{_\s}\bigoplus_{uv = w}
    \left\langle \mathrm{ISS}_{p,r}^\mathbb{S}(z), u \right\rangle_{\!\s} 
    \odot_\s \left\langle \mathrm{ISS}_{r,q}^\mathbb{S}(z), v \right\rangle_{\!\s},
  \end{align}
  or, equivalently, using the non-commutative concatenation product on $\S\llangle A \rrangle$,
  \begin{align}
    \label{lem:chen2}
    \mathrm{ISS}_{p,r}^\S(z)
    \mathrm{ISS}_{r,q}^\S(z)
    =
    \mathrm{ISS}_{p,q}^\S(z).
  \end{align}
\end{lemma}

\begin{remark}
  \label{rem:chen}
  This allows, for $p < r < q$, to
  calculate $\mathrm{ISS}_{p,q}^\mathbb{S}(z)$ from
  $\mathrm{ISS}_{p,r}^\mathbb{S}(z)$ and $\mathrm{ISS}_{r,q}^\mathbb{S}(z)$.
  However, in contrast to the situation over a ring,
  we cannot calculate
  $\mathrm{ISS}_{r,q}^\mathbb{S}(z)$ from $\mathrm{ISS}_{p,q}^\mathbb{S}(z)$ and
  $\mathrm{ISS}_{p,r}^\mathbb{S}(z)$.
  This is due to the fact that semiring addition, $\oplus_\s$, is not invertible.
  We return to this point in \Cref{sec:algo}.
\end{remark}
\begin{example}
\label{ex:chen2}
One example was already given in \Cref{ex:chen}.
Consider now again the min-plus semiring $\S=\R_\mp$ (here $d=1$ corresponding to the single letter alphabet $A'=\{1\}$).
Then
\begin{align*}
  &\Big\langle \ISS^{\R_\mp}_{p,q}(z), [\w1^7][\w1^1][\w1^4] \Big\rangle  
  =
  \sideset{}{_{\mp}}\bigoplus_{p < i_1 < i_2< i_3 < q+1} z_{i_1}^{\odot_{\mp} [\w1^7]} \odot_{\mp} z_{i_2}^{\odot_{\mp} [\w1^7]} \odot_{\mp} z_{i_3}^{\odot_{\mp} [\w1^4] }\\
  &= \min_{p < i_1 < i_2  < i_3 \le q } \{ 7 z_{i_1} + z_{i_2} + 4 z_{i_3} \}\\
  &=
  \min \Big\{
    \min_{ p < i_1 < i_2 < i_3 \le r}  \{ 7 z_{i_1} + z_{i_2} + 4 z_{i_3} \},
    \min_{ p < i_1 < i_2 \le r}  \{ 7 z_{i_1} + z_{i_2}\}
    +
    \min_{r < i_3 \le q} \{ 4 z_{i_3} \},\\
  &\qquad
    \min_{ p < i_1 \le r}  \{ 7 z_{i_1}\}
    +
    \min_{r < i_2 < i_3 \le q} \{ z_{i_2} + 4 z_{i_3} \},
    \min_{ r < i_1 < i_2 \le q}  \{ 7 z_{i_1} + z_{i_2} + 4 z_{i_3} \}
  \Big\} \\
  &=
  \Big\langle \ISS^{\R_\mp}_{p,r}(z), [\w1^7][\w1^1][\w1^4] \Big\rangle
  \oplus_{{\mp}}
  \left( \Big\langle \ISS^{\R_\mp}_{p,r}(z), [\w1^7][\w1^1] \Big\rangle \odot_\mp \Big\langle \ISS^{\R_\mp}_{p,r}(z), [\w1^4] \Big\rangle \right) \\
  &\qquad
  \oplus_{{\mp}}
  \left( \Big\langle \ISS^{\R_\mp}_{p,r}(z), [\w1^7] \Big\rangle \odot_\mp \Big\langle \ISS^{\R_\mp}_{p,r}(z), [\w1^1][\w1^4] \Big\rangle \right)
  \oplus_{{\mp}}
  \Big\langle \ISS^{\R_\mp}_{p,r}(z), [\w1^7][\w1^1][\w1^4] \Big\rangle
\end{align*}

\end{example}

\begin{proof} We now show \eqref{eq:chen} by a direct calculation
  \begin{align*}
    \left\langle \mathrm{ISS}_{p,q}^\mathbb{S}(z), w \right\rangle
    &= \sideset{}{_\s}\bigoplus_{p < j_1 < j_2 < \cdots < j_k\le q} z_{j_1}^{\odot_\s w_1} \odot_\s \dots \odot_\s z_{j_k}^{\odot_\s w_k}\\
    &=
    \sideset{}{_\s}\bigoplus_{p < r < j_1 < j_2 < \dots < j_k\le q} z_{j_1}^{\odot_\s w_1} \odot_\s \dots \odot_\s z_{j_k}^{\odot_\s w_k}\\
    &\oplus_\s
    \sideset{}{_\s}\bigoplus_{p < j_1 \le r < j_2 <\dots < j_k\le q} z_{j_1}^{\odot_\s w_1} \odot_\s \dots \odot_\s z_{j_k}^{\odot_\s w_k} \\
    &
    \oplus_\s
    \dots
    \\
    &\oplus_\s \sideset{}{_\s}\bigoplus_{p < j_1 < j_2 < \dots < j_{k-1} \le r < j_k\le q} z_{j_1}^{\odot_\s w_1} \odot_\s \dots \odot_\s z_{j_k}^{\odot_\s w_k}\\
    &\oplus_\s
    \sideset{}{_\s}\bigoplus_{p < j_1 < j_2 < \dots < j_{k-1} < j_k \le r < q} z_{j_1}^{\odot_\s w_1} \odot_\s \dots \odot_\s z_{j_k}^{\odot_\s w_k}\\
    &=
    \sideset{}{_\s}\bigoplus_{uv=w}
    \left\langle \mathrm{ISS}_{p,r}^\mathbb{S}(z), u \right\rangle 
    \odot_\s \left\langle \mathrm{ISS}_{r,q}^\mathbb{S}(z), v \right\rangle.
  \end{align*}
\end{proof}

From summation by parts \eqref{sumbyparts} extended to iterated $\mathbb{S}$-sums we deduce the multiplicativity of $\mathrm{ISS}^\mathbb{S}(z)$ over the quasi-shuffle product \eqref{Squasishuff} on $\mathbb{S}\langle A\rangle$. 

\begin{lemma}[Multiplicativity]
  \label{lem:qsIdentity}
For $w,v \in \mathbb{S}\langle A\rangle$ and $z \in \mathbb{S}_{\zero_\s}^{d,\N_{\ge 1}}$
 \begin{align*}
   \left\langle \ISS_{s,t}^\S(z),w \right\rangle \odot_\s \left\langle \ISS_{s,t}^\S(z),v\right\rangle
   =
   \left\langle \ISS_{s,t}^\S(z), w \qs v \right\rangle
  \end{align*}
\end{lemma}

\begin{example}
\label{ex:qs2}
\Cref{ex:qs} in the current notation reads as
\begin{align*}
  &\Big\langle \ISS^{\R_\mp}(z), [\w1^1] \Big\rangle
  \odot_\s
  \Big\langle \ISS^{\R_\mp}(z), [\w1^7][\w1^4] \Big\rangle \\
  &\quad=
  \Big\langle \ISS^{\R_\mp}(z), [\w1^1][\w1^7][\w1^4] \Big\rangle
  \oplus_\s
  \Big\langle \ISS^{\R_\mp}(z), [\w1^8][\w1^4] \Big\rangle
  \oplus_\s
  \Big\langle \ISS^{\R_\mp}(z), [\w1^7][\w1^1][\w1^4] \Big\rangle\\
  &\qquad\qquad
  \oplus_\s
  \Big\langle \ISS^{\R_\mp}(z), [\w1^7][\w1^5] \Big\rangle
  \oplus_\s
  \Big\langle \ISS^{\R_\mp}(z), [\w1^7][\w1^4][\w1^1] \Big\rangle.
\end{align*}
\end{example}

\begin{proof}
  We first note that for a word $x \in A^*$ and a letter $a \in A$,
  \begin{align}
    \label{eq:halfShuffle}
    \bigopluss_{s < i < j} \left\langle \ISS_{s,i-1}^{\S}(z), x \right\rangle \odots z^{\odot_\s a}_{i} = \left\langle \ISS_{s,j-1}^{\S}(z), x a \right\rangle,
  \end{align}
  which follows directly from the definition of $\TSS$.

  We prove the Lemma by induction on the sum $q = \len(w)+\len(v)$ of the lengths of the words.
  It is trivially true for $q = 0,1$.
  Suppose that \(q\ge 2\).  Let it be true up to arbitrary $q-1$ and assume $\len(w)+\len(v) = q$.
  Without loss of generality, we assume that both \(w=w_1\cdots w_\ell\) and \(v=v_1\cdots v_k\) are nonempty since otherwise there is nothing to prove.
  For $i=s+1,\dots,t$,
define
  \begin{align*}
    f_i &\coloneqq \left\langle \ISS_{s,i-1}^{\S}(z), v_1 \cdots v_{k-1} \right\rangle \odots z^{\odot_\s v_k}_{i} \\
    g_i &\coloneqq \left\langle \ISS_{s,i-1}^{\S}(z), w_1 \cdots w_{\ell-1} \right\rangle \odots z^{\odot_\s w_\ell}_{i}.
  \end{align*}
  Note that for a fixed $j$ one has
  \begin{align*}
    \bigopluss_{s < i < j}
    f_i
    =
    \left\langle \ISS_{s,j-1}^{\S}(z), v \right\rangle \qquad
    \bigopluss_{s < i < j}
    g_i
    =
    \left\langle \ISS_{s,j-1}^{\S}(z), w \right\rangle.
  \end{align*}
  
  Then
  \begin{align}
    \label{eq:threeTerms}
    &\left\langle \ISS_{s,t}^{\S}(z), v \right\rangle
    \odots
    \left\langle \ISS_{s,t}^{\S}(z), w \right\rangle 
    =
    \left( \bigopluss_{s < i < t + 1} f_i \right)
    \odots
    \left( \bigopluss_{s < j < t + 1} g_i \right) \notag\\
    &=
    \left( \bigopluss_{s < i < j < t + 1} f_i \odots g_j \right)
    \opluss
    \left( \bigopluss_{s < j < i < t + 1} f_i \odots g_j \right)
    \opluss
    \left( \bigopluss_{s < i < t + 1} f_i \odots g_i \right),
  \end{align}
  where the last equality follows from summation-by-parts, \eqref{sumbyparts}.
  Now, the first term in the last equality is equal to
  \begin{align*}
    &\bigopluss_{s < i < j < t+1}
    f_i
    \odots \left\langle 
    \ISS_{s,j-1}^{\S}(z), w_1 \cdots w_{\ell-1} \right\rangle \odots z^{\odot_\s w_\ell}_{j} \\
    &=\bigopluss_{s < j < t+1} \left\langle \ISS_{s,j-1}^{\S}(z), v \right\rangle \odots \left\langle 
    \ISS_{s,j-1}^{\S}(z), w_1 \cdots w_{\ell-1} \right\rangle \odots z^{\odot_\s w_\ell}_{j} \\
    &=
    \bigopluss_{s < j < t+1} \left\langle \ISS_{s,j-1}^{\S}(z), v \qs (w_1 \cdots w_{\ell-1}) \right\rangle \odots z^{\odot_\s w_\ell}_{j} \\
    &=
    \left\langle \ISS_{s,t}^{\S}(z), ( v \qs (w_1 \cdots w_{\ell-1}) ) w_\ell \right\rangle,
  \end{align*}
  where we used \eqref{eq:halfShuffle}
  and the induction hypothesis (since $\len(v)+\len(w_1\cdots w_{\ell-1}) = q-1$).
  Analogously, for the second term, we obtain
  \begin{align*}
    \left\langle \ISS_{s,t}^{\S}(z), (v_1 \dots v_{k-1}) \qs w \right\rangle.
  \end{align*}
  The last term is equal to
  \begin{align*}
    &\bigopluss_{s < i < t+1} \left\langle \ISS_{s,i-1}^{\S}(z), v_1 \cdots v_{k-1} \right\rangle \odots z^{\odot_\s v_k}_{i}  \odots \left\langle \ISS_{s,i-1}^{\S}(z), w_1 \cdots w_{\ell-1} \right\rangle \odots z^{\odot_\s w_\ell}_{i} \\
    &=\bigopluss_{s < i < t+1} \left\langle \ISS_{s,i-1}^{\S}(z), (v_1 \cdots v_{k-1})\qs (w_1 \cdots w_{\ell-1}) \right\rangle \odots z^{\odot_\s v_k}_{i}  \odots z^{\odot_\s w_\ell}_{i} \\
    &=\bigopluss_{s < i < t+1} \left\langle \ISS_{s,i-1}^{\S}(z), (v_1 \cdots v_{k-1})\qs (w_1 \cdots w_{\ell-1}) \right\rangle \odots z^{\odot_\s [v_k\ w_\ell]}_{i} \\
    &=\left\langle \ISS_{s,t}^{\S}(z), \bigl((v_1 \cdots v_{k-1})\qs (w_1 \cdots w_{\ell-1})\bigr) [v_k\ w_\ell]\right\rangle.
  \end{align*}
  where we used \eqref{eq:halfShuffle}
  and the induction hypothesis.
  
  Combining those terms, we get, via \eqref{Squasishuff}, that
  \begin{align*}
    &\left\langle \ISS_{s,t}^{\S}(z), v \right\rangle
    \odots
    \left\langle \ISS_{s,t}^{\S}(z), w \right\rangle \\
    &=
    \left\langle \ISS_{s,t}^{\S}(z), ( v \qs (w_1 \cdots w_{\ell-1}) ) w_\ell + ((v_1 \cdots v_{k-1})\qs w) v_k \right. \\
    &\qquad 
     +\bigl( (v_1 \cdots v_{k-1} v_{k-1})\qs (w_1 \cdots w_{\ell-1}) \bigr)[v_k w_\ell] \Big\rangle \\
    &=
    \left\langle \ISS_{s,t}^{\S}(z), v \qs w \right\rangle.
  \end{align*}
\end{proof}

\section{Quasisymmetric expressions over a semiring}
\label{sec:qsym}

The aim of this section is to study the coefficients, i.e. the iterated sums used to define the $\TSS$ in \eqref{S-ISS}, as formal power series expressions.%
\footnote{We speak of ``quasisymmetric expressions'' and not of ``quasisymmetric functions'',
since, over a semiring, polynomial expressions cannot be identified with polynomial functions \cite{CK2016}.}
Analogous to the classical case, this results in the definition of the notion of quasisymmetric expressions defined over a semiring.
These are formal series with coefficients in $\mathbb{S}$ which have a particular symmetry property defined below.
When considered over a commutative ring, their siblings form the well-studied Hopf algebra of quasisymmetric
functions with the monomial quasisymmetric functions as one of many bases \cite{Ges1984,luoto2013introduction,MR1995}. As we shall see, working over a semiring leads to rather minor changes compared to the classical theory of quasisymmetric functions.
This stems from the fact that most properties only rely on the \emph{index set} (i.e., the totally ordered set of integers). However, it turns out that the monomial basis is the only reasonable one, see  \Cref{rem:monomialBasis}.

In the following, we denote by 
\begin{align*}
  \S \lbracket X_1, X_2, X_3, \dots\rbracket
\end{align*}
the commutative $\S$-semialgebra of \textbf{formal power series expressions} in commuting ordered indeterminates
$X \coloneqq \{X_1,X_2,X_3,\ldots \}$ with coefficients in $\mathbb{S}$.  We write monomials in these variables in the usual form
\begin{align*}
  X_{s_1}^{\alpha_1} \cdots X_{s_n}^{\alpha_n}, \quad n \ge 0,\quad \alpha_1, \dots, \alpha_n \ge 1,
\end{align*}
but note that this is -- of course -- \emph{just a formal expression},
so that we might as well have written $X_{s_1}^{\odot_\s \alpha_1} \odot_\s \dots \odot_s X_{s_n}^{\odot_\s \alpha_n}$.%
Which, although more consistent with previous notation, we abstain from to reduce notational clutter.
The \textbf{degree} of such a monomial is defined by the sum of the exponents, $|X_{s_1}^{\alpha_1} \cdot \dots \cdot X_{s_n}^{\alpha_n}| \coloneqq \alpha_1 + \dots + \alpha_n$.
Similar to the power series semialgebra in noncommuting variables of the previous section,
elements $P \in \S \lbracket X_1, X_2, X_3, \dots\rbracket $
can be considered as formal power series
\begin{align*}
  P = \sum_{m} c_m m,
\end{align*}
where $c_m \in \S$ and the sum is over formal commutative monomials in the indeterminates  $X_1,X_2,\dots$. The linear structure follows as for the case of noncommutative variables and the multiplicative structure is induced from the product of formal monomials (Cauchy product).
We shall write $P(m) \coloneqq c_m$.
By small abuse of notation, we let $\S \lbracket X_1, X_2, X_3, \dots\rbracket $ contain \textbf{only power series of bounded degree}, i.e., for $P \in \S \lbracket X_1, X_2, X_3, \dots\rbracket $ there is $N \ge 0$ such that for all monomials $m$ with $|m| \ge N$, $P(m) = \zero_\s$.
The subset with power series of finite support is denoted $\S[X_1,X_2,\dots]$ and forms the space of \textbf{formal polynomial expressions}.

\begin{definition}
  An element $P \in \mathbb{S} \lbracket X_1, X_2, X_3, \dots\rbracket $
  is a \textbf{quasisymmetric expression}
  if for all $n \ge 1$, $\alpha_1, \ldots, \alpha_n \ge 1$,
  $0 < s_1 < \dots < s_n$ and
  $0 < t_1 < \dots < t_n$
  the coefficients
  of the monomials
  \begin{align*}
    X_{s_1}^{\alpha_1} \cdot \dots \cdot X_{s_n}^{\alpha_n}\, \text{ and }\,
    X_{t_1}^{\alpha_1} \cdot \dots \cdot X_{t_n}^{\alpha_n},
  \end{align*}
  coincide.

  Define the \textbf{monomial quasisymmetric expression}
  indexed by $\alpha = (\alpha_1, \dots, \alpha_k) \in \N_{\ge 1}^k$, $k \ge 0$
  as
  \begin{align*}
    M_{\alpha} \coloneqq
    \sum_{1 \le t_1 < \dots < t_k < +\infty} X_{t_1}^{\alpha_1} \cdot \dots \cdot X_{t_k}^{\alpha_k},
  \end{align*}
  with the convention $M_{()} \coloneqq \one_\s$, the constant quasisymmetric expression.
\end{definition}

\begin{lemma}
  The space of all quasisymmetric expressions
  is a sub-semialgebra of the commutative $\S$-semialgebra $\mathbb{S}\lbracket X_1, X_2, \dots\rbracket $.
  We denote it by $\mathrm{QSym}_\s$.
\end{lemma}
\begin{proof}
  Let $Q,P$ be quasisymmetric expressions,
  let $n\ge 1$ and
  $\alpha_1, \ldots, \alpha_n \ge 1$,
  $0 < s_1 < \dots < s_n$ and
  $0 < t_1 < \dots < t_n$.
  Then, since the product of power series expressions is the Cauchy product,
  \begin{align*}
    (Q P)(X_{s_1}^{\alpha_1} \cdot \dots \cdot X_{s_n}^{\alpha_n})
    &=
    \bigopluss_{k=0}^n
    \bigopluss_{ 1 \le i_1 < \dots < i_k \le n}
      Q( X_{s_{i_1}}^{\alpha_{i_1}} \cdot \dots \cdot X_{s_{i_k}}^{\alpha_{i_k}} )
      P( X_{s_{j_1}}^{\alpha_{j_1}} \cdot \dots \cdot X_{s_{j_{n-k}}}^{\alpha_{j_{n-k}}} ),
  \end{align*}
  where the inner sum is over all $I = \{i_1 < \dots < i_k \} \subset [n]$
  and $\{j_1 < \dots < j_{n-k}\} \coloneqq [n] \setminus I $ is the complement in $[n]$.
  Since both $Q$ and $P$ are quasisymmetric, this sum is equal to
  \begin{align*}
    &\bigopluss_{k=0}^n
    \bigopluss_{ 1 \le i_1 < \dots < i_k \le n}
      Q( X_{t_{i_1}}^{\alpha_{i_1}} \cdot \dots \cdot X_{t_{i_k}}^{\alpha_{i_k}} )
      P( X_{t_{j_1}}^{\alpha_{j_1}} \cdot \dots \cdot X_{t_{j_{n-k}}}^{\alpha_{j_{n-k}}} ) \\
    &\quad=
    (Q P)(X_{t_1}^{\alpha_1} \cdot \dots \cdot X_{t_n}^{\alpha_n}).
  \end{align*}
  Hence, the space of quasisymmetric expressions is indeed closed under multiplication.
\end{proof}

\begin{remark}
  \begin{enumerate}
    \item 
      We can naturally evaluate a formal monomial at values in a commutative semiring, e.g.~for semiring elements $z_1,z_2 \in \S$,
      \begin{align*}
        X_1^3 X_2^5\evaluatedAt{X_1 = z_1, X_2 = z_2} = z_1^{\odot_\s 3} \odot_\s z_2^{\odot_\s 5}.
      \end{align*}
    The iterated sums in the definition of the $\TSS$, \eqref{S-ISS}, then
    amount, in the one-dimensional case, to evaluation of specific monomial quasisymmetric expression,
    \begin{align*}
      \left\langle \ISS^\S(z), [\w1^{\alpha_1}] \cdots [\w1^{\alpha_k}] \right\rangle_\s = M_\alpha\evaluatedAt{X_1 = z_1, X_2 = z_2, \dots}.
    \end{align*}

  \item
    See \cite[Remark 3.5]{DET2020} for a straightforward extension to ``multidimensional'' quasisymmetric expressions.

  \end{enumerate}
\end{remark}

\begin{example}
  The simplest, non-trivial quasisymmetric expression
  is
  \begin{align*}
    M_{(1)}(m) \coloneqq
    \begin{cases}
        \one_\s & \qquad \text{ if } m = X_i \qquad \text{ for some $i$ } \\
        \zero_\s & \qquad \text{ else }
    \end{cases},
  \end{align*}
  or, written as formal sum,
  \begin{align*}
    M_{(1)}
    =
    \sum_{0 < i < +\infty} X_{i}
    =
    \sum_{0 < i < +\infty} \one_\s X_{i}.
  \end{align*}

  Another example is given by
  \begin{align*}
    M_{(1,2)}=\sum_{0 < i_1 < i_2 < +\infty} X_{i_1} X_{i_2}^{2}.
  \end{align*}
\end{example}

\begin{definition}
  We say that a family $v_i, i \in I$, of vectors in a $\S$-semimodule $M$
  is a \textbf{basis} if
  for every $w \in M$ there exist
  unique coefficients $\lambda_i \in \S, i \in I$, all but finitely
  many equal to zero, with
  \begin{align}
    \label{eq:basis}
    \sum_{i \in I} \lambda_i v_i = w.
  \end{align}
\end{definition}

\begin{remark}
  There exist different concepts of linear independence in semimodules. The reader is referred to \cite{AGG2009} for an overview.%
  \footnote{We thank the anonymous referee in pointing out an error in an earlier version on this subtle topic.}
\end{remark}

We then have, as expected, that the monomial quasisymmetric expressions are a basis for $\mathrm{QSym}_\s$.
\begin{proposition}
  The family $M_\alpha$ of monomial quasisymmetric expressions a basis for
  $\mathrm{QSym}_\s$.
\end{proposition}
\begin{proof}
  To show that every element of $\mathrm{QSym}_\s$ can be written
  as a linear combination of monomial quasisymmetric expressions,
  let $Q \in \mathrm{QSym}_\s$.
  If $Q$ is the zero power series, we are done.
  Otherwise,
  take a monomial $m = X_{t_1}^{\alpha_1} \cdot \dots \cdot X_{t_k}^{\alpha_k}$ in $Q$ with non-zero coefficient $c \in \S$
  (i.e., considering $Q$ as a function on monomials, $Q(m)=c$).
  Then
  \begin{align*}
    Q = Q' + c M_{\alpha}
  \end{align*}
  with $Q' \in \mathrm{QSym}_\s$.
  Since $Q$ has finite degree we can repeat this finitely many times to
  see that $Q$ is a linear combination of monomial quasisymmetric expressions.

  The uniqueness of coefficients follows from
  the fact that for $\alpha\not=\beta$ the supports of
  $M_\alpha$ and $M_\beta$ are disjoint.
\end{proof}

\begin{remark}
  \label{rem:monomialBasis}
  The space of quasisymmetric functions over a commutative ring
  has several other relevant linear bases, for example, the \emph{fundamental basis},
	\cite[eq. (2.13)]{MR1995}. Regarding the latter, in the semiring setting, the corresponding fundamental quasisymmetric expressions, $F_\beta \in \mathrm{QSym}_\s$, can be defined as linear combinations of monomial quasisymmetric expressions 
  $$
    F_\alpha = \sum_{\beta \ge \alpha} M_\beta,
  $$
  where $\le$ is the following ``refinement'' partial order:
  \begin{align*}
    (\beta_1,\dotsc,\beta_m)\le(\alpha_1,\dotsc,\alpha_n)
  \end{align*}
  if and only if there is a sequence of integers \(k_1,\dotsc,k_m\) with \(k_j\ge 1\) and \(k_1+\dotsb+k_m=n\), and such that \(\beta_i=\alpha_{j}+\alpha_{j+1}+\dotsb+\alpha_{j+k_i-1}\) for all \(i=1,\dotsc,m\).
  For example, \( (4,4)\le(3,1,2,2) \).

	It turns out that those expressions are closed under multiplication, i.e., they form a sub-semialgebra of $\mathrm{QSym}_\s$ .
    This subalgebra is strictly smaller (and might be interesting to investigate further).
    Indeed, the classical inversion formula (between the $M_\alpha$ and $F_\beta$) implies, for instance, that $M_{(2)}$ can not be expressed in terms of fundamental quasisymmetric expressions in the semiring setting (having no minus operation available)\footnote{We thank Darij Grinberg for pointing this out to us.}.   

    In fact, more holds, as the following statement shows.
    \begin{lemma}
      Let $\S$ be a semiring without additive inverses,
      that is $a \oplus b = \zero$ implies $a = b = \zero$,
      and zero-divisor free.
      Then: the monomial basis is the \emph{only} basis of $\mathrm{QSym}_\s$.
    \end{lemma}
    \begin{proof}
      Let $G_i$, $i \in I$, be another basis and let $M_\alpha$ be some monomial basis element.
      Then we can write
      \begin{align*}
        M_\alpha = \sum_{j=1}^n c_j \odot G_{i_j}, 
      \end{align*}
      with $c_j \in \S_\mp \setminus \{\zero_\mp\}$.
      Let $m$ be any monomial not appearing in $M_\alpha$.
      Then
      \begin{align*}
        \zero_\mp = \sideset{}{_\mp}\bigoplus_{j=1}^n c_j \odot_\mp G_{i_j}(m).
      \end{align*}
      By assumption, this implies $c_j \odot_\mp G_{i_j}(m) = \zero_\mp, j=1,\dots,n$.
      Since $\S$ is zero-divisor free and the $c_j$ are not equal to $\zero$,
      $m$ does not appear in any of the $G_{i_j}$.
      Hence $n=1$ and $c_1 \odot G_{i_1} = M_\alpha$.
      Hence the basis $(G_i)_i$ contains, up to multiplicative factors,
      the monomial basis. Since this subset already forms a basis,
      the basis $(G_i)_i$ is equal, up to multiplicative factors,
      to the monomial basis.
    \end{proof}
\end{remark}

\subsection{Invariance to inserting zeros}
\label{ssec:invzero}

We will now show that $\mathrm{QSym}_\s$ can be characterized by invariance to ``inserting zeros''.
For $n\ge 1$ define the commutative $\mathbb{S}$-semialgebra morphism
\begin{align*}
  \insertZero_n\colon \mathbb{S}\lbracket X_1,X_2,\dots\rbracket  \to \mathbb{S}\lbracket X_1,X_2,\dots\rbracket ,
\end{align*}
induced multiplicatively from the following map on $X_1, X_2, \dots$
\begin{align*}
  \insertZero_i[X_j]
  &=
  \begin{cases}
    X_i & i < j \\
    \zero_\s   & i = j \\
    X_{i-1} & i > j.
  \end{cases}
\end{align*}
If we consider $P \in \mathbb{S}\lbracket X_1,X_2,\dots\rbracket $ as $\S$-valued function on monomials, $m = X_{t_1}^{ \alpha_1}  \cdots  X_{t_n}^{ \alpha_n}$,
(where $t_1 < \cdots < t_n$ and $\alpha_i > 0$)
this implies that $\insertZero_i[P](m)$ is equal to
\begin{center}
  \vspace{-1em}
\resizebox{.95\linewidth}{!}{
  \begin{minipage}{\linewidth}
\begin{align*}
  \begin{cases}
    P(m) & t_n < i \\
    \zero_\s & i \in \{t_1+1,\dots,t_n+1\} \\
    P\left(X_{t_1+1}^{\alpha_1} \cdots X_{t_n+1}^{\alpha_n} \right) & i \le t_1 \\
    P\left(X_{t_1}^{\alpha_1} \cdot \dots \cdot X_{t_{k-1}}^{\alpha_{k-1}} X_{t_k+1}^{\alpha_k} X_{t_{k+1}+1}^{\alpha_{k+1}}\cdots X_{t_n+1}^{\alpha_n} \right) & t_{k-1} < i \le t_k, k\in\{2,\dots,n\}
  \end{cases}
\end{align*}
\end{minipage}
}
\end{center}

\begin{example}
  \begin{align*}
    \insertZero_9[  X_2 X_6^{7} X_8^{5} ] &= X_2 X_6^{7} X_8^{5} \\
    \insertZero_8[  X_2 X_6^{7} X_8^{5} ] &= \zero_\s \\
    \insertZero_1[  X_2 X_6^{7} X_8^{5} ] &= X_1 X_5^{7} X_7^{5} \\
    \insertZero_3[  X_2 X_6^{7} X_8^{5} ] &= X_2 X_5^{7} X_7^{5}.
  \end{align*}
\end{example}

\begin{theorem}
  \label{thm:insertZero1}
  A power series expression $P \in \S\lbracket X_1,X_2,\dots\rbracket $ is in $\operatorname{QSym}_\s$ if and only if
  \begin{align*}
    \insertZero_n P = P \qquad \forall n \ge 1.
  \end{align*}
\end{theorem}
\begin{example}
  If $P$ is invariant in the prescribed sense,
  then for any monomial $m$
  \begin{align*}
    P(m)
    =
    \Big( (\insertZero_1)^2 \insertZero_{2} P \Big)( m ).
  \end{align*}
  We apply this to $m = X_{3}^{7} X_{5}$ to get
  \begin{align*}
    P(X_{1}^{7} X_{2})
    =
    \Big( (\insertZero_1)^2 \insertZero_{2} P \Big)( X_{1}^{7} X_{2} )
    =
    P( X_3^{7} X_{5} ).
  \end{align*}
  Since the time points $3,5$ were arbitrary, the coefficients of all monomials $X_{t_1}^{7} X_{t_2}$, $1 \le t_1 < t_2 < +\infty$, must coincide.
\end{example}
\begin{proof}
  $\Rightarrow$:
  Immediate, since the morphism zero preserves the exponents and the order of the variables involved.

  $\Leftarrow$:
  Let $n \ge 1$, $\alpha_1, \ldots, \alpha_n$,
  $0 < t_1 < \dots < t_n$ be given.
  We then have
  \begin{align*}
    P\left( X_{1}^{\alpha_1} \cdots X_{n}^{\alpha_n} \right)
    &=
    \Big( (\insertZero_1)^{t_1-1} (\insertZero_{2})^{t_2-t_1-1} \cdots (\insertZero_{n})^{t_n-t_{n-1}-1} P \Big)\left( X_{t_1}^{\alpha_1} \cdots X_{t_n}^{\alpha_n} \right) \\
    &=
    P\left( X_{t_1}^{\alpha_1} \cdots X_{t_n}^{\alpha_n} \right).
  \end{align*}
  Since $n$, $\alpha_1, \ldots, \alpha_n$ and $t_1, \ldots, t_n$ were arbitrary this shows that $P$ is quasisymmetric.
\end{proof}

From Theorem \ref{thm:insertZero1} we get the following consequence.
\begin{corollary}
\label{cor:inv1}
    $\TSS(z)_{0,\infty}$ is invariant to inserting $\zero_\s$ into the time series $z \in \mathbb{S}_{\zero_\s}^{d,\N_{\ge 1}}$.
\end{corollary}

More precisely, let \(\tau_n\colon\mathbb S_{\zero_\s}^{d,\N_{\ge 1}}\to\mathbb S_{\zero_\s}^{d,\N_{\ge 1}}\) be the map that inserts \(\zero_\s\) at position \(n\) into a time series, that is,
\[
	\tau_n(z)_m = \begin{dcases}z_m&m<n\\\zero_\s&m=n\\z_{m-1}&m>n\end{dcases}.
\]
With this notation, \Cref{cor:inv1} means that if \(z'=\tau_{n_1}\circ\dotsb\circ\tau_{n_k}(z)\) for some set of integers \(0\le n_1\le\dotsb\le n_k\) then
\[
	\TSS(z')_{0,\infty}=\TSS(z)_{0,\infty}.
\]
In fact, under some restrictions, a converse statement is also true.
Recall that a semiring $\S$
\begin{itemize}
  \item has the \textbf{cancellation property} if $a\odot c=b\odot c$ and $c\not=\zero$ imply $a = b$,
  \item is \textbf{zero-divisor free} if $a\odot b = \zero$ implies $a=\zero$ or $b=\zero$.
\end{itemize}
We remark that
if $\S$ is a ring then both notions coincide and $\S$ is then also called an integral domain.
For semirings, the cancellation property is the strictly stronger notion of the two as the following examples show.
\begin{example}
  The bottleneck semiring $(\R \cup \{\pm \infty\},\max,\min,-\infty,+\infty)$ does \emph{not}
  satisfy the cancellation property, as for example $3 \odot 7 = 7 = 4 \odot 7$,
  but it is zero-divisor free: $\min\{ a, b \} = -\infty$ implies that $a=-\infty$ or $b=-\infty$.

  \bigskip

  As another example consider
  the semiring of ideals of a commutative ring $R$
  \begin{align*}
    \S &\coloneqq \{ I \subset R : I \text{ ideal} \} \\
    a \oplus b &\coloneqq ( a \cup b ) = \{ x + y : x \in a, y \in b \} \\
    a \odot b  &\coloneqq ( x \cdot y: x\in a, y \in b ) = \{ x_1 \cdot y_1 + \dots + x_n y_n : x_i \in a, y_i \in b, n \ge 1 \} \\
    \zero_\s &\coloneqq \emptyset \\
    \one_\s &\coloneqq R.
  \end{align*}
  Its multiplicative semigroup has been extensively studied, see for example \cite{gilmer1992multiplicative}.

  If $R$ is an integral domain, then $\S$ is clearly zero-divisor free.
  It does, in general, \emph{not} have the cancellation property.
  Indeed in $R = \Z[2i]$ we have
  \begin{align*}
    (2) \odot (2,2i) = (2i) \odot (2,2i),
  \end{align*}
  but $(2) \not= (2i)$.

  Another example is $R = \mathbb{C}[x,y]$ with
  \begin{align*}
    C &= (x,y) \\
    A &= (x^2+y^2,xy,z)     \\ 
    B &= (x^2,y^2,xy,z) = C^2. 
  \end{align*}
  Then
  \begin{align*}
    A\odot C = B\odot C.
  \end{align*}
  It turns out that $\S$ has the cancellation property if and only if
  every ideal in $R$ is locally a regular principal ideal,
  \cite{anderson1997characterization}.
\end{example}

In the following, given \(z\in\S^{d,\N_{\ge1}}_{\zero_\s}\) we denote by \(\check{z}\) its \emph{compression}, obtained by deleting all the zero entries in \(z\) before it is eventually constant.
\begin{theorem}
  \label{thm:CPandZDF}
    If $\S$ has the cancellation property, then
	the identity $\TSS(z)_{0,\infty}=\TSS(z')_{0,\infty}$ implies that \(\check{z}=\check{z}'\) for all \(z,z'\in\mathbb S^{d,\N_{\ge 1}}_{\zero_\s}\).

	If the identity \(\TSS(z)_{0,\infty}=\TSS(z')_{0,\infty}\) implies that \(\check{z}=\check{z}'\) for all \(z,z'\in\mathbb S^{d,\N_{\ge 1}}_{\zero_\s}\)
    then $\S$ is zero-divisor free.
\end{theorem}
\begin{remark}
  It is an open question whether the second part of the theorem can be strengthened to
  imply the cancellation property.
\end{remark}
\begin{proof}
	Suppose $\S$ has the cancellation property and let \(z,z'\in\S^{d,\N_{\ge 1}}_{\zero_\s}\) be such that \(\TSS(z)_{0,\infty}=\TSS(z')_{0,\infty}\).
	Without loss of generality, we assume that \(d=1\). In the general case, the conclusion follows after applying the following argument along each dimension.

	Let \(N\in\N_{\ge1}\) be the largest integer such that \(\langle\TSS(z)_{0,\infty},[\w1]^N\rangle\neq 0\).
	Then there exists an increasing sequence \(1\le i_1<\dotsb<i_N\) of integers such that \(z_{i_j}\ne\zero_\s\), and \(z_n=\zero_\s\) for all \(n>i_N\).
	In particular, we must have
	\[
		\langle\TSS(z)_{0,\infty},[\w1]^N\rangle=z_{i_1}\odot\dotsm\odot z_{i_N}.
	\]
	We call the integer \(N\) the \emph{effective length} of \(z\), and we note that it must be the stay the same for \(z'\), although the exact sequence of timestamps might differ.

	Now, by assumption, we have that
	\begin{equation}
	\label{eq:tw.id}
		z_{i_1}\odot\dotsm\odot z_{i_N}=z'_{j_1}\odot\dotsb\odot z'_{j_N}.
	\end{equation}
	Moreover, by picking out the coefficient of the word \([\w{11}][\w1]^{N-1}\) we see that
	\[
		z_{i_1}^{\odot 2}\odot z_{i_2}\odot\dotsb\odot z_{i_N}=(z'_{j_1})^{\odot 2}\odot z'_{j_2}\odot\dotsb\odot z'_{j_N}=z'_{j_1}\odot z_{i_1}\odot\dotsb\odot z_{i_N}.
	\]
	Since \(\S\) is has the cancellation property and \(z_{i_1}\neq\zero_\s\), this implies that \(z^{\phantom\prime}_{i_1}=z'_{j_1}\).

    Inserting this equality into \cref{eq:tw.id} and using again the fact that $\S$ has the cancellation property,  we conclude that
	\begin{equation}
\label{eq:tw.id2}
		z_{i_2}\odot\dotsm\odot z_{i_N}=z'_{j_2}\odot\dotsm\odot z'_{j_N}.
	\end{equation}
	Testing now against the word \([\w1][\w{11}][\w1]^{N-2}\) implies, by a similar argument, that \(z^{\phantom\prime}_{i_2}=z'_{j_2}\).
	The argument then continues by replacing this new identity into \cref{eq:tw.id} to obtain an equality similar to \cref{eq:tw.id2} but with fewer terms, and then picking out words of the form \([\w1]^k[\w{11}][\w1]^{N-k-1}\).
	This procedure clearly terminates after \(N\) steps, and the output is that \(z^{\phantom\prime}_{i_k}=z'_{j_k}\) for all \(k=1,\dotsc,N\), i.e., \(\check{z}=\check{z}'\).

    Assume now that $\S$ is not zero-divisor free.
	Then, there exist \(a,b\in\S\setminus\{\zero_\s\}\) such that \(a\odot b=0\).
    If $a\not= 0$ the time series \(z=(a,b,\zero_\s,\dotsc)\) and \(z'=(b,a,\zero_\s,\dotsc)\) are
    such that \(\check{z}\neq\check{z}'\)
    but with \(\TSS(z)_{0,\infty}=\TSS(z')_{0,\infty}\).
	Indeed, it is easy to check that for the words \([\w1]\), \([\w{11}]\) and \([\w1][\w1]\) both coefficients coincide.
	The only other potentially non-zero coefficients in both signatures are those associated with words of the form \([\w1^k][\w1]\) or \([\w1][\w1^k]\) for some \(k\ge 2\).
    In any case, we have that the coefficients are either \(a^{\odot k} \odot b\) or \(b^{\odot k} \odot a\) which all vanish.

    Assume now $a=b$, i.e. $a^2 = \zero_\s$.
    If $a+ a \not= 0$ and $a + a \not= a$ we can repeat the argument with
    \(z=(a,a+a,\zero_\s,\dotsc)\) and \(z'=(a+a,a,\zero_\s,\dotsc)\).
    If $a+a = 0$, then the signature
    of $z=(a,a+a,\zero_\s,\dotsc)$ is trivial, although $\check{z}$ is not trivial.

    Finally, if $a+a = a$ then
    \(z=(a,\zero_\s,\dotsc)\) and \(z'=(a,a,\zero_\s,\dotsc)\) have the same iterated-sums signature
    although $\check{z}\not=\check{z}'$.
    This finishes the proof.
\end{proof}

\section{Time warping invariants in an idempotent semiring}
\label{sec:timeWarpingInvariants}
\newcommand\idem{\mathrm{idem}}

\Cref{ex:logTimeSeries} together with \Cref{cor:inv1}
shows one way to obtain time warping invariants
of a real-valued time series.
This does not cover the invariant \eqref{eq:min} though.

Since $\R \subset \R \cup \{+\infty\}$, and since the tropical semiring is idempotent
we can also calculate $\TSS(z)$ on a real-valued time series that is eventually constant.
Recall that $\TSS(z)=\TSS_{0,\infty}(z)$. Since
\begin{align*}
  \left\langle \TSS(z), [\w{1}] \right\rangle = \min_i z_i,
\end{align*}
this includes the invariant \eqref{eq:min}.
But, as is quickly seen, most coefficients
are \emph{not} invariant to time warping.
To wit,
\begin{align}
  \label{eq:strict}
  \left\langle \TSS(z), [\w{1}][\w{1}] \right\rangle = \min_{i_1 < i_2} \{ z_{i_1} + z_{i_2} \},
\end{align}
gives, for,
\begin{align*}
  z  &= (1,-3,2,2,\dots) \\
  z' &= (1,-3,-3,2,2,\dots),
\end{align*}
the values $-2$ and $-6$ respectively.

It turns out that if we change the strict inequality over points in time in \eqref{eq:strict} into a weak or non-strict inequality,
namely
\begin{align*}
  \min_{i_1 \le i_2} \{ z_{i_1} + z_{i_2} \},
\end{align*}
then we \emph{do} get a time warping invariant.
In this section we would like to spell out how this works in general.

\begin{definition}
Assume that $\S$ is an idempotent semiring.
Let $z$ be a time series with values in $\S$ that is eventually constant.
We define for $1 \le s < t \le +\infty$,
\begin{align}
  \label{S-ISS-weak}
  \left\langle \ISS_{s,t}^{\S,\idem}(z), w\right\rangle \coloneqq \sideset{}{_\s}\bigoplus_{s < j_1 \le j_2 \le \dots \le j_k < t+1} z_{j_1}^{\odot_\s w_1} \odot_\s \cdots \odot_\s z_{j_k}^{\odot_\s w_k},
\end{align}
where the possibly infinite sum is well-defined, since $\S$ is idempotent and $z$ is eventually constant.
As before, we write $\ISS^{\S,\idem}(z) = \ISS_{0,+\infty}^{\S,\idem}(z)$.
\end{definition}

The following lemma is immediate.
\begin{lemma}
  \label{lem:idempotentTimeWarping}
  $\ISS_{s,t}^{\S,\idem}(z)$ is invariant to \textbf{time warping}. That is, define for $n \ge 1$ the time series $\tau_n(z)$ as
  \begin{align*}
    \tau_n(z)_j \coloneqq
    \begin{cases}
      z_j     & j \le n \\
      z_{j-1} & j > n.
    \end{cases}
  \end{align*}
  Then, for all $n \ge 1$:
  \begin{align*}
    \ISS_{s,t}^{\S,\idem}(\tau_n(z))
    =
    \ISS_{s,t}^{\S,\idem}(z).
  \end{align*}
\end{lemma}

\begin{lemma}
  \label{lem:idemShuffle}
  $\ISS_{s,t}^{\S,\idem}(z)$ is a shuffle character, i.e.
  \begin{align*}
    \left\langle \ISS_{s,t}^{\S,\idem}(z), v \right\rangle
    \odot_\s
    \left\langle \ISS_{s,t}^{\S,\idem}(z), w \right\rangle
    =
    \left\langle \ISS_{s,t}^{\S,\idem}(z), v \shuffle w \right\rangle.
  \end{align*}

\end{lemma}
\begin{example}
Using, for example, the computation in \eqref{shufflesum}, we see that 
  \begin{align*}
    &\left\langle \ISS_{s,t}^{\S,\idem}(z), [\w1^7][\w1^3] \right\rangle
    \odot_\s
    \left\langle \ISS_{s,t}^{\S,\idem}(z), [\w1^5] \right\rangle \\
    &=
       \left\langle \ISS_{s,t}^{\S,\idem}(z), [\w1^7][\w1^3][\w1^5] + [\w1^7][\w1^5][\w1^3] + [\w1^5][\w1^7][\w1^3] \right\rangle,
  \end{align*}
  where we used idempotency of $\opluss$.
\end{example}

\begin{proof}
  The proof follows analogously to the one of \Cref{lem:qsIdentity}.
  Owing to idempotency, for $f_i,g_i \in \S$, $i=s+1,\dots,t$,
  \eqref{sumbyparts} becomes
  \begin{align*}
    \left( \bigopluss_{s < i \le t} f_i \right)
    \odots
    \left( \bigopluss_{s < j \le t} g_j \right)
    =
    \left( \bigopluss_{s < i \le j \le t} f_i \odots g_j \right)
    \opluss
    \left( \bigopluss_{s < j \le i \le t} f_i \odots g_j \right).
  \end{align*}
  This leads to the last term in \eqref{eq:threeTerms} not being present
  and hence to a shuffle product instead of a quasi-shuffle product.
\end{proof}

We note that, in the tropical semiring, $\ISS^{\R_\mp,\idem}$ is very degenerate, in the sense that, in the one-dimensional case,
\begin{align*}
  \left\langle \ISS^{\R_\mp,\idem}(z), [\w1^{a_1}] \cdots [\w1^{a_n}] \right\rangle
  =
  \left\langle \ISS^{\R_\mp,\idem}(z), [\w1^{a_1 + \cdots + a_n}] \right\rangle.
\end{align*}

\newcommand{\llrrparen}[1]{
  \left(\mkern-4mu\left(#1\right)\mkern-4mu\right)}
  \newcommand\laurentIdem{\llrrparen{\idem}}
To get a more interesting object we can allow powers in $\Z \setminus \{0\}$ (instead of just $\N_{\ge 1}$),
e.g.
\begin{align*}
  \left\langle \ISS^{\R_\mp,\laurentIdem}(z), [\w1^{-3}][\w1^5] \right\rangle
  &\coloneqq
  \sideset{}{_{\R_\mp}}\bigoplus_{0 < j_1 \le j_2} z_{j_1}^{\odot_{\R_\mp} -3} \odot_{\R_\mp} z_{j_2}^{\odot_{\R_\mp} 5} \\
  &=
  \min_{0 < j_1 \le j_2} \{ -3 z_{j_1} + 5 z_{j_2} \}.
\end{align*}

\begin{proposition}
  Let $\R_\mp$ be the tropical semiring. Define for $w \in A^*$, where $A = \Z\setminus\{0\}$,
  \begin{align*}
		\left\langle \ISS^{\R_\mp,\laurentIdem}_{s,t}(z), w \right\rangle
    \coloneqq \sideset{}{_{\R_\mp}}\bigoplus_{s < j_1 \le j_2 \le \dots \le j_k \le t} z_{j_1}^{\odot_{\R_\mp} w_1} \odot_{\R_\mp} \cdots \odot_{\R_\mp} z_{j_k}^{\odot_{\R_\mp} w_k},
  \end{align*}
  Then:
  \begin{enumerate}
    \item $\ISS^{\R_\mp,\laurentIdem}$ is a shuffle character.
    \item $\ISS^{\R_\mp,\laurentIdem}$ satifies Chen's identity.
    \item $\ISS^{\R_\mp,\laurentIdem}$ is time warping invariant.
  \end{enumerate}
\end{proposition}
\begin{proof}
  The shuffle property follows as in \Cref{lem:idemShuffle}.
  The time-warping invariance is immediate, as in \Cref{lem:idempotentTimeWarping}.
  Chen's identity follows from
  the partition of the time interval
  \begin{align*}
    \{ p < i_1 \le \dots \le i_k \le q \}
    &=
    \{ p \le r < i_1 \le i_2 \le \dots \le i_k \le q \}
    \dot\cup
    \{ p < i_1 < r \le i_2 \dots \le i_k \le q \} \\
    &\qquad
    \dot\cup
    \dots
    \{ p < i_1 \le i_2 \le \dots \le i_k \le r \le q \},
  \end{align*}
  for $p < r \le q$.
\end{proof}

\begin{remark}
  The iterated-sums signature over a field of characteristic $0$
  is, via the Hoffman exponential, in bijection to a certain \emph{iterated-integrals} signature,
  \cite[Theorem 5.3]{DET2020}.
  The iterated-sums signature satisfies a quasi-shuffle identity, whereas the 
  iterated-integrals signature is a shuffle character.
  In fact, there is a whole family of signature-like maps, indexed by \(\theta\in(-1,1)\) obtained by composing
  the iterated sums signature with some linear transformation \(A_{\theta\to 1}\) which generalize
  Hoffman's exponential (it being the case \(\theta=0\)), see \cite[Remark 2.3]{DET2020a}.

  When working over an idempotent semiring, however, only the cases \(\theta=-1\) and \(\theta=0\)
  are well defined, and both maps coincide.
\end{remark}

\section{Algorithmic considerations}

\label{sec:algo}

As in the setting of iterated sums over a field,
Chen's identity, \Cref{lem:chen},
leads to an algorithm
to calculate
\begin{align*}
  \left\langle \TSS_{0,T}(z), w \right\rangle,
\end{align*}
for a word $w$ of length $\len(w)=k$, which is of computational cost $\O(T \cdot k)$ 
and storage cost $\O(T \cdot k)$.

\begin{algorithm2e}[H]

    \SetKwData{Prev}{prev}\SetKwData{Curr}{curr}
    \SetKwFunction{len}{len}\SetKwFunction{FindCompress}{FinCompress}
    \SetKwProg{Fn}{Function}{}{end}\SetKwFunction{FTSS}{TSS}%

    \KwData{\ \ $z$: time series, $w$: word }
    \KwResult{An array containing $\langle \TSS_{0,t}(z), w \rangle,\ t=1,\ldots,\len(z)$}

    \Fn{\FTSS{$T,z,w$}}{

    $T$ $\coloneqq$ $\len(z)$

    \eIf{$\len(w)$ $\ge 2$}
    {
       \Prev $\coloneqq$ $\FTSS(z,w_1 \cdots w_{k-1})$
    }
    {
       \Prev $\coloneqq$ $(\one_\s,\ldots,\one_\s)$
    }

    \BlankLine
    \Curr $\coloneqq$ $[\zero_\s]$

    \For{$t = 1$ \texttt{\normalfont\KwTo} $T$}{
      \textsf{curr[t]} $\coloneqq$ \textsf{curr[t-1]} $\oplus_\s$ \textsf{prev[t-1]} $\odot_\s\ z_t^{\odot_\s w_k}$
    }
    \KwRet{\Curr}
    }
    \caption{Calculation of $\TSS$ via Chen's identity / dynamic programming}
\end{algorithm2e}

\bigskip

One sees that the algorithm actually yields
\begin{align}
  \label{eq:actually}
  \left\langle \TSS_{0,t}(z), w_1 \cdots w_i \right\rangle,
  \qquad \text{ for \emph{all} $i=1, \ldots, k$, $t=1, \ldots, T$}.
\end{align}
\newcommand\K{{\mathbb{K}}}
Over a \emph{field} $\K$, owing to the group-likness of $\ISS^\K$ implying the identity
\begin{align*}
  \ISS^\K_{s,t}(z) = \ISS^\K_{0,s}(z)^{-1} \ISS^\K_{0,t}(z),
\end{align*}
one obtains, with computational preparation cost $\O(T \cdot k)$
and storage cost $\O(T \cdot k)$ (i.e.~the cost of computing respectively storing \eqref{eq:actually}, for $\S=\K$),
the values
\begin{align*}
  \Big\langle \ISS^\K_{s,t}(z), w_1 \cdots w_i \Big\rangle,
\end{align*}
for \emph{arbitrary} $s,t=1, \ldots, T$, $i=1, \ldots, k$, 
at query cost $\O(k)$.%
\footnote{
  For a field $\K$, the bialgebra $\K\langle A\rangle$ possesses an antipode $\alpha$,
  and then
  \begin{align*}
    \Big\langle \ISS^\K_{s,t}(z), w_1 \cdots w_i \Big\rangle,
    =
    \Big\langle \ISS^\K_{0,s}(z)^{-1} \ISS^\K_{0,t}(z), w_1 \cdots w_i \Big\rangle
    =
    \Big\langle \ISS^\K_{0,s}(z) \otimes \ISS^\K_{0,t}(z), (\alpha \otimes \id) \Delta ( w_1 \cdots w_i ) \Big\rangle,
  \end{align*}
  involves $\O(k)$ many queries to 
  $\ISS^\K_{0,s}(z)$ and $\ISS^\K_{0,t}(z)$.
}
Such a cheap access to iterated sums over \emph{arbitrary} sub-intervals
is key to some applications in machine learning.

Over a general semiring, $\TSS_{0,s}(z)^{-1}$ is not defined.
But, a slightly less efficient calculation for all subintervals simultaneously is still possible:
\begin{proposition}
  There is an algorithm that, for given a word $w$ of length $k$ and time horizon $T$,
  %
  obtains
  \begin{align*}
    \left\langle \TSS_{s,t}(z), w_1 \cdots w_i \right\rangle \text{ and }\quad
    \left\langle \TSS_{s,t}(z), w_i \cdots w_k \right\rangle,
  \end{align*}
  for all $s,t=1, \ldots, T$, $i=1, \ldots, k$ at ``query'' cost $\O( \log(T)^k )$.
  
  The computational ``preparation'' cost (explained in the proof)
  for this algorithm is $\O(k \log(T) T)$ and the storage cost is $\O(k T)$.
\end{proposition}
\begin{proof}
  For ease of notation we assume $T = 2^N$ for some $N$.
  Let $w = w_1 \cdots w_k$ be given and define $W \coloneqq \{w_1 \cdots w_i, w_i \cdots w_k \mid i = 1, \ldots, k\}$.
  Note that $\# W = 2k$.
  \begin{enumerate}

    \item (\textbf{Preparation})
      Calculate
      \begin{align*}
        I_{j,n}
        \coloneqq
        \Big\langle \TSS_{j \cdot 2^n, (j+1) \cdot 2^n}(z), v \Big\rangle,\ n = 0, \dots, N,\ j = 0, \ldots ,2^{N-n} - 1,\ v \in W,
      \end{align*}
      which, for a fixed $j,n$ has computational cost $\O( \# W\ 2^n )$
      and hence in total gives computational ``preparation'' cost $\O( \# W\ N 2^N ) = \O(k \log(T) T)$
      and storage cost $\O(2^{N+1}) = \O(T)$.

    \item (\textbf{Query})
      For given $1 \le s < t \le T=2^N$
      pick $I_{j_1,n_1} < \cdots < I_{j_q,n_q}$ with
      \begin{align*}
        \bigcup_{i=1}^q I_{j_i,n_i} = (s,t],
      \end{align*}
      and such that $q$ is minimal.
      This can, for example, be done by the exploration of a corresponding binary tree,
      with cost $\O(N) = \O(\log(T))$.
      Note that $q \le 2 N = 2 \log(T)$.
      Then, for $v \in W$,
      \begin{align*}
        \Big\langle \TSS_{s,t}(z), v \Big\rangle
        =
        \Big\langle \prod_{i=1}^q \TSS_{j_i \cdot 2^{n_i}, (j_i+1) \cdot 2^{n_i}}(z), v \Big\rangle
        =
        \sum_{v_1 \cdots v_q= v} \bigodotss_{i=1}^q \Big\langle \TSS_{j_i \cdot 2^{n_i}, (j_i+1) \cdot 2^{n_i}}(z),  v_i \Big\rangle.
      \end{align*}
      The $q$-th iterated ``deconcatenation''
      $\sum_{v_1 \cdots v_q= v} \cdots$
      contains $\O( {k+q-1 \choose q-1} )$ terms.
      The latter expression is monotonically increasing in $q$ and hence bounded by
      \begin{align*}
        \O( {k+2\log(T)-1 \choose 2\log(T)-1} ),
      \end{align*}
      yielding a ``query'' cost of 
      \begin{align*}
        \O( {k+2\log(T)-1 \choose 2\log(T)-1} + \log(T) )
        =
        \O( \log(T)^k ).
      \end{align*}

  \end{enumerate}

\end{proof}

\begin{figure}[!ht]
    \centering
        \includegraphics[width=0.6\textwidth]{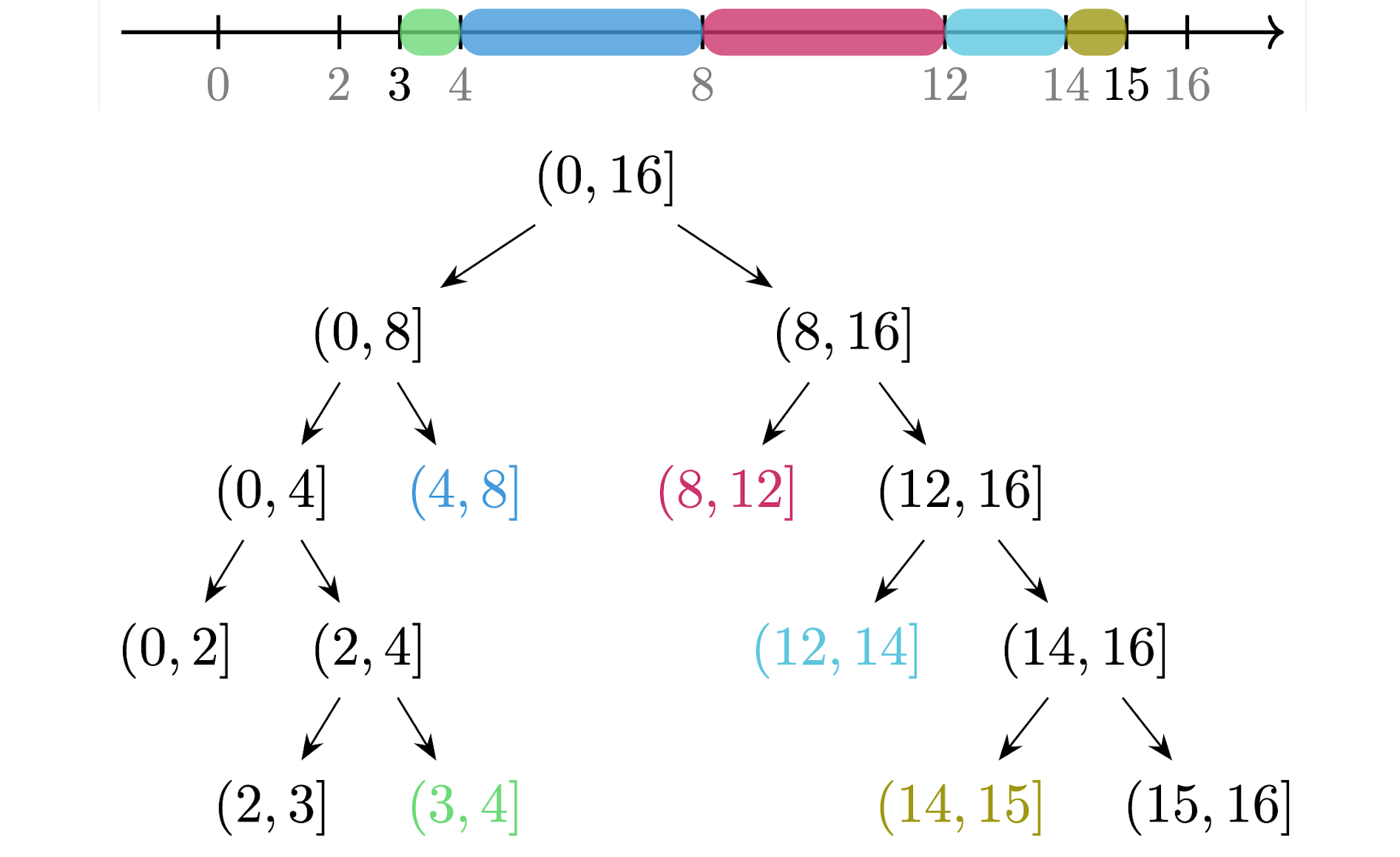}
        \caption{Example of the coarsest dyadic partition of $(3,15]$}
        \label{fig:binaryTree}
\end{figure}

\begin{example}
  Consider the time horizon $T=16=2^4$
  and
  $s=3, t=15$. There is a (unique) coarsest partition of the interval $(3,15]$ using only
  dyadic intervals
  $(j \cdot 2^n, (j+1) \cdot 2^n]$, namely
  \begin{align*}
    (3,15] = (3,4] \dot\cup (4,8] \dot\cup (8,12] \dot\cup (12,14] \dot\cup (14,15],
  \end{align*}
  see also \Cref{fig:binaryTree}.
  We then obtain
  \begin{align*}
      \Big\langle \TSS_{3,15}(z), w \Big\rangle
      &=
      \Big\langle \TSS_{3,4}(z), w \Big\rangle
      \oplus_\s
      \Big\langle \TSS_{4,8}(z), w \Big\rangle
      \oplus_\s
      \Big\langle \TSS_{8,12}(z), w \Big\rangle \\
      &\qquad
      \oplus_\s
      \Big\langle \TSS_{12,14}(z), w \Big\rangle
      \oplus_\s
      \Big\langle \TSS_{14,15}(z), w \Big\rangle.
  \end{align*}

\end{example}

\section{Conclusion}
\label{sec:concl}

In \eqref{S-ISS} we introduced the iterated-sums signature, $\TSS(z)$, over a commutative semiring $\S$.
It stores all iterated sums (taken in the semiring) of a multidimensional time series $z=(z_1,z_2,\ldots)$ with entries
$z_i \in \S^d$.
As in the case over commutative rings, this object satisfies
Chen's identity (\Cref{lem:chen})
which here as well allows for efficient computation.
It also satisfies the quasi-shuffle identity (\Cref{lem:qsIdentity})
that is, it behaves like a group-like element.
Unlike for the usual ISS there is no proper Hopf algebra structure available here
and it is in general \emph{not} possible to take the logarithm
of the signature.

In the one-dimensional case over commutative rings,
the entries of the iterated-sums signature correspond to the evaluation
of certain formal power series, namely quasisymmetric functions.
Here, this is also true (\Cref{sec:qsym}), though it is more appropriate to speak of quasisymmetric \emph{expressions},
since polynomial expressions over a semiring are \emph{not} in one-to-one correspondence with
polynomial functions.

In order to explicitly cover the expression \eqref{eq:min} from the Introduction,
we looked at the special case of idempotent semirings in \Cref{sec:timeWarpingInvariants}.

Calculation of $\TSS_{0,t}(z)$ with a \emph{fixed} starting time $0$,
can be done at the same (linear) cost as in the case over commutative rings.
Owing to the lack of additive inverses,
the calculation of $\TSS(z)$ over \emph{arbitrary} subintervals
is more intricate.
In \Cref{sec:algo} we provided an algorithm,
that incurrs an additional logarithmic cost factor.

\subsection{Open questions}

\begin{itemize}

  \item
    The iterated-sums signature over the reals has a close connection
    to discrete control theory as explored in reference \cite{gray2020nonlinear}.
    In the setting of the max-plus semiring: 
    \begin{quote}
      \emph{Is there a relation to discrete control theory in that semiring \cite{Cohen1995,CL2008,komenda2018}?\footnote{We thank one of the referees for pointing us to chapter 5 in \cite{CL2008}, where the representation of timed-event graphs by algebras of tropical operators is discussed, as well as to Theorem 5.4 in reference \cite{Cohen1995}, where the representation of the input-output map of a discrete event system by a tropical Volterra series is given. Further studies of these results in the light of our approach seem to be worthwhile.}}
    \end{quote}

  \item
    The ISS over a commutative ring contains,
    owing to the quasi-shuffle identity, many redundant entries.
    Working with the log-signature
    removes these redundancies.
    Over a general commutative semiring, we cannot take
    the logarithm of the signature,
    so an open question is 
    \begin{quote}
      \emph{How to extract the ``minimal'' information contained in the signature?}
    \end{quote}

  \item
    As seen in \Cref{rem:chen}, owing to the lack of additive inverses, Chen's identity
    only works in ``one direction''.
    \begin{quote}
      \emph{Is there a way (with maybe a larger object) of obtaining a general Chen's identity?}
    \end{quote}

  \item
    Multidimensional time series are explicitly covered by the present work.
    Just as over the reals, this amounts to projecting the time series
    to coordinates before calculating the iterated-sums.

    In the semiring setting, a more interesting approach seems possible.
    Indeed, one can turn a multidimensional real-valued time series
    into a one-dimensional semiring-valued time series.
    One example is via the map
    \begin{align*}
      \R^d &\to \text{ bounded convex polytopes } \\
      x &\mapsto \{ x \}.
    \end{align*}
    The resulting time series can then be considered in the semiring of polytopes, point 8. in \Cref{ex:semirings}.
    One can hope for tractable calculation, using the relation to the algebra
    of polynomials, \cite[Theorem 2.25]{pachter2005algebraic}.

  \item
    Chen's identity, \Cref{lem:chen}, applied
    to time points $0,t,t+1$ reads
    as
    \begin{align*}
      \left\langle \TSS_{0,t+1}(z), w \right\rangle_{\!\s}
      &=
      \left\langle \TSS_{0,t}(z), w \right\rangle_{\!\s} 
      \oplus_\s
      \left(
      \left\langle \TSS_{0,t}(z), w_1 \cdots w_{n-1} \right\rangle_{\!\s}
      \odot_\s
      \left\langle \TSS_{t,t+1}(z), w_n \right\rangle_{\!\s} \right) \\
      &=
      \left\langle \TSS_{0,t}(z), w \right\rangle_{\!\s} 
      \oplus_\s
      \left(
      \left\langle \TSS_{0,t}(z), w_1 \cdots w_{n-1} \right\rangle_{\!\s}
      \odot_\s z_t^{\odot_\s w_n}  \right),
    \end{align*}
    where we use the notation of \eqref{S-ISS}.
    This allows for an iterative calculation of this value,
    with total cost $\O(n \cdot t)$.
    This can be seen as a special case of dynamic programming.
    \begin{quote}
      \emph{Is there a deeper connection to the dynamic programming literature?}
    \end{quote}
    
  \item
    The iterated-integrals signature has been investigated
    from the perspective of algebraic geometry in 
    \cite{amendola2019varieties}.
    \begin{quote}
      \emph{Is there interesting tropical algebraic geometry,
      that can be done on the objects introduced in this work?}
    \end{quote}

    \item
    In  \Cref{sec:category} we recall
    how category theory provides
    an organized view on semirings and their modules.
    In a push for further generality,
    \begin{quote}
      \emph{Is it possible to categorify the iterated-sums signature?}
    \end{quote}
    The category of polynomial functors
    \cite{spivak2020poly}
    and the abstract view on time series in
    \cite{schultz2020dynamical}
    might prove  useful in this regard.

  \item
    What connection can be drawn to the literature
    on temporal logic \cite{pnueli1977temporal}
    and Allen's temporal logic \cite{allen1983maintaining}?

\end{itemize}

\appendix
\clearpage


\section{Categorial view on semirings}
\label{sec:category}

The aim of this section is to give a brief overview of the categorical setting for semirings, semimodules, etc.
Good references on category theory are \cite{ML1971,Rie2016} (see also \cite{brandenburg2016einfuhrung} (in German)).
For the particularities of monoidal categories, we refer to \cite[Section 4.1]{Bra2014} and \cite{Mar2009}.

Recall that a \textbf{monoidal category} is a category \(\mathsf C\) with a bifunctor
\(\otimes\colon\mathsf C\times\mathsf C\to\mathsf C\), and an object \(1\in\mathsf C\) called the
\emph{unit} such that there exist natural isomorphisms
\[ 
	( (-)\otimes (-))\otimes (-)\cong(-)\otimes ( (-)\otimes (-)),
	\quad 1\otimes(-)\cong(-),
	\quad (-)\otimes 1\cong(-)
\]
and satisfy some consistency relations. Essentially, this means that there is a notion of ``tensor product'' internal to the category. A monoidal category is \textbf{symmetric} if furthermore it is endowed with a ``braiding'' or ``twisting'' natural isomorphism \(\tau_{X,Y}\colon X\otimes Y\to Y\otimes X\) such that
\( \tau_{Y,X}\circ\tau_{X,Y}=\operatorname{id}_{X\otimes Y}\).
Examples of symmetric monoidal categories include \(\mathsf{Vect}_k\) for any field \(k\) and
\(\mathsf{Mod}_R\) for any commutative ring \(R\).
In both cases, \(\otimes\) corresponds to the internal tensor product.

In any monoidal category, the notion of monoid makes sense. A \textbf{monoid} on a
monoidal category \(\mathsf C\) is an object \(M\) in \(\mathsf C\) together with two arrows
\(\mu\colon M\otimes M\to M\) and \(u\colon 1\to M\) satisfying an associativity and unitality
condition \cite[Section VII.3]{ML1971}.
Additionally, in a symmetric monoidal category,
one can also impose a commutativity constraint and obtain commutative monoids.
As an example, monoids in \(\mathsf{Vect}_k\) correspond to algebras over vector spaces.
Dually, a comonoid in \(\mathsf C\) is a monoid in the dual (or opposite) category \(\mathsf C^{\mathrm{op}}\).%
\footnote{One may also consider bimonoids and Hopf monoids.}
As monoids on the category of vector spaces correspond to algebras, comonoids in \(\mathsf{Vect}_k\)
correspond to coalgebras.
The category of monoids in \(\mathsf C\) is denoted by \(\operatorname{Mon}(\mathsf C)\).%
\footnote{The arrows are given by arrows in \(\mathsf C\) that respect the monoid structure, \cite[Definition 1.2.9]{Mar2009}.}
Likewise, the category of commutative monoids in \(\mathsf C\) is denoted by
\(\operatorname{CMon}(\mathsf C)\).

\begin{proposition}[Theorem VII.3.2 in \cite{ML1971}]
  \label{prop:free}
    Let \(\mathsf C\) be a monoidal category with countable coproducts
    and assume that for each $A \in \mathsf C$
    the functors $- \otimes A, A \otimes -$ preserve countable coproducts.
    Then the forgetful functor \(U\colon\operatorname{Mon}(\mathsf C)\to\mathsf C\)
    has a left adjoint \(F\colon\mathsf C\to\operatorname{Mon}(\mathsf C)\). On an object \(X\) in
    \(\mathsf C\), the underlying object of \(F(X)\) is
    \[ U(F(X))=\coprod_{n=0}^\infty X^{\otimes n} \]
    in \(\mathsf C\), with the monoidal structure given by the tensor product.
\end{proposition}

\begin{example}
	In the category \(\mathsf{Set}\), coproducts correspond to disjoint unions, i.e., for any countable sequence \((A_n)\) of sets,
	\[
		\coprod_{n=0}^\infty A_n\cong\bigsqcup_{n=0}^\infty A_n.
	\]
	The category \(\mathsf{Set}\) also possesses a monoidal structure given by the cartesian product, that is, \(A\otimes B\coloneqq A\times B\).
	Hence,
	\[
		F(X) = \bigsqcup_{n=0}^\infty X^{\times n}
	\]
	can be identified with the free monoid over the set \(X\), since monoids over \(\mathsf{Set}\) are classical monoids.
\end{example}
\begin{example}
	In the category $\mathsf{Vect}_k$, coproducts correspond to direct sums, i.e., for any countable sequence \((V_n)\) of vector spaces over \(k\),
	\[
		\coprod_{n=0}^\infty V_n\cong\bigoplus_{n=0}^\infty V_n.
	\]
	In particular
  \begin{align*}
    F(X) = T(X) = \bigoplus_{n \ge 0} X^{\otimes n},
  \end{align*}
	is the tensor algebra over $X$ since monoids over \(\mathsf{Vect}_k\) correspond to unital \(k\)-algebras.
\end{example}

Suppose that \(\mathsf C\) is a symmetric monoidal category and let \(R\) be a commutative monoid object in
\(\mathsf C\).
A \textbf{left \(R\)-module} is an object \(M\) in \(\mathsf C\) with an arrow \(\mu_M\colon R\otimes
M\to M\) defining an action of \(R\) on \(M\).
One can also define right \(R\)-modules in an obvious way, but since \(R\) is commutative both
notions coincide and we just call them \textbf{\(R\)-modules}.
The category of \(R\)-modules is denoted by \(\mathsf{Mod}_R\).%
\footnote{With obvious definition of morphisms, compare \cite[Definition 1.2.11]{Mar2009}.}

A nice example of symmetric monoidal category is \(\mathsf{Ab}\), the category of abelian groups.
The tensor product on \(\mathsf{Ab}\) is obtained form the cartesian (or direct) product of abelian
groups by modding out the relations \( (a_1,b)+(a_2,b)-(a_1+a_2,b) \) and
\((a,b_1)+(a,b_2)-(a,b_1+b_2)\). We denote it by \(\otimes\) as usual.
The unit for this tensor product is the group of integers with addition.
Monoid objects over this category correspond to rings, that is,
\(\operatorname{Mon}(\mathsf{Ab})=\mathsf{Ring}\).
Given an object $R$ in \(\operatorname{CMon}(\mathsf{Ab})\), that is, a commutative ring, the notion of
\(R\)-module corresponds to the usual notion of a module over a ring.

For any given commutative monoid object \(R\) in a symmetric monoidal category \( (\mathsf C,\otimes,1)\), the category of modules
\(\mathsf{Mod}_R\) is also a symmetric monoidal category when endowed with the tensor product
\(M\otimes_R N\) defined as the \textbf{coequalizer} of the two maps \(M\otimes R\otimes
N\rightrightarrows M\otimes N\) given by the action of \(R\) on \(M\) and \(N\).
The unit for this tensor product is \(R\) seen as a module over itself.
Hence, the complete data is \((\mathsf{Mod}_R,\otimes_R,R)\).

Now, we consider the category of monoids \(\operatorname{Mon}(\mathsf{Mod}_R)\).
We call objects in this category \textbf{algebras over \(R\)}.
Likewise, comonoids in \(\mathsf{Mod}_R\) are called \textbf{coalgebras}.

\begin{theorem}[Proposition 1.2.14 in \cite{Mar2009}]
    Let \(\mathsf C\) be a symmetric monoidal category.
    If \(\mathsf C\) is either complete or cocomplete, then so are \(\operatorname{CMon}(\mathsf C)\)
    and \(\mathsf{Mod}_R\) for any commutative monoid \(R\) in \(\mathsf C\).
    \label{thm:complete}
\end{theorem}

\subsection{The category of semirings}
\label{ssec:categsemi}

We apply the above construction to the category \(\mathsf{CMon}\coloneqq\operatorname{CMon}(\mathsf{Set})\) of commutative monoids.
This is a symmetric monoidal category, with product given by the tensor product of commutative monoids \cite[Appendix B]{APT2018}.
\begin{proposition}
  \label{prop:complete}~

  \begin{enumerate}
    \item 
    The category \(\mathsf{CMon}\) is complete and cocomplete.

    \item
    For $R$ an object in $\operatorname{CMon}(\mathsf{CMon})$,
    the category $\operatorname{Mod}_R$ is complete and cocomplete.
  \end{enumerate}
\end{proposition}
\begin{proof}
  1. A commutative monoid is nothing else than a commutative monoid object in the symmetric monoidal category
    \((\mathsf{Set},\times,\{*\})\) which is known to be complete and cocomplete \cite[Theorem 3.2.6, Proposition 3.5.1]{Rie2016}.
    Therefore, \(\mathsf{CMon}\) is also complete and cocomplete by \Cref{thm:complete}.

    2. This follows from point 1. and \Cref{thm:complete}.
\end{proof}

\begin{proposition}
  \label{prop:preserves}~

  \begin{enumerate}
    \item 
    In \(\mathsf{CMon}\), for each $A \in \mathsf{CMon}$
    the functors $- \otimes A$ and $A \otimes -$ preserve countable coproducts.

    \item
    For $R$ an object in $\operatorname{CMon}(\mathsf{CMon})$,
    the tensor product in $\operatorname{Mod}_R$ preserves countable coproducts.
  \end{enumerate}
\end{proposition}
\begin{proof}
    1. By \cite[Appendix B]{APT2018}, \(\mathsf{CMon}\) is closed. In particular, this means
    that \(-\otimes A\) is left adjoint to \(\operatorname{Hom}(A,-)\). Therefore, \(-\otimes A\)
    preserves colimits \cite[Theorem 4.5.3]{Rie2016}, and in particular coproducts.
    By symmetry of the tensor product, the same is true for \(A\otimes -\).

    2. By \cref{prop:complete}, \(\mathsf{CMon}\) is complete. Therefore, by \cite[Theorem
    4.1.10]{Bra2014}, for any commutative monoid
    \(R\) in \(\mathsf{CMon}\), \(\operatorname{Mod}_R\) is a cocomplete symmetric monoidal category
    with unit \(R\) and tensor product \(\otimes_R\).
\end{proof}

The category of commutative monoids \(\operatorname{CMon}(\mathsf{CMon})\) corresponds to
\textbf{commutative semirings},
i.e., commutative rings without negative elements.%
\footnote{Also called \emph{rigs}, rings without \textbf{n}egative elements.}
For a fixed commutative semiring \(R\), objects in the category \(\mathsf{Mod}_R\) are known as
\textbf{semimodules}.
The notions of \textbf{semialgebra}, \textbf{semi-coalgebra} and \textbf{Hopf semialgebra} follow
(as monoid, comonoid, and Hopf monoid in \(\mathsf{Mod}_R\)).

\begin{theorem}
  Let $R$ be a commutative semiring.

  \begin{enumerate}
  \item
      There exists a left adjoint $F\colon \operatorname{Set} \to \mathsf{Mod}_R$ to the forgetful functor.
  $F(D)$ is the \textbf{free $R$-module over $D$}.

  \item
  There exists a left adjoint $F'\colon \mathsf{Mod}_R \to \operatorname{Mon}(\mathsf{Mod}_R)$
  to the forgetful functor.
  $F'(X)$ is the \textbf{free semialgebra over $X$}.
  \end{enumerate}
\end{theorem}
\begin{proof}
  Using Proposition \ref{prop:complete} and Proposition \ref{prop:preserves}
  we can apply Proposition \ref{prop:free}.
\end{proof}

$F(D)$ corresponds to the free $\S$-semimodule over a set $D$ indicated in Section \ref{sec:tss}. It is \textbf{free} in the following sense. For every map $\phi$ from $D$ into an $\mathbb{S}$-semimodule $M$ there exists a unique $\mathbb{S}$-semimodule morphism $\Phi\colon \mathbb{F} \to M$ such that the following diagram commutes.
\begin{center}
\begin{tikzcd}
D \arrow[rd, "\phi"] \arrow[r, hook] & \mathbb{F} \arrow[d, "\Phi", dotted] \\
                                     & M                          
\end{tikzcd}
\end{center}
\subsubsection{Symmetric powers}
Let \(\mathsf C\) be a symmetric monoidal category and \(R\) a commutative monoid object.
Recall that the category \(\mathsf{Mod}_R\) is itself a symmetric monoidal category with \(R\) as the unit.
Given an object \(M\) in \(\mathsf{Mod}_R\), the symmetric group \(\mathbb S_n\) acts naturally on \(M^{\otimes_Rn}\).
In fact, the action is induced by the braiding map corresponding to \(\otimes_R\).

Suppose that \(\mathsf C\) has finite colimits. The symmetric power \(S^nM\) is then defined to be the coequalizer of the action of \(\mathbb S_n\) on \(M^{\otimes_Rn}\).
By definition, this means that \(S^nM\) has the following universal property: every \emph{symmetric morphism}\footnote{That is, an arrow invariant to the action of the symmetric group.} \(f\colon M^{\otimes_Rn}\to N\) extends uniquely to a morphism \(\tilde f\colon S^nM\to N\).
In particular, there are canonical epimorphisms \(S^nM\otimes S^mM\twoheadrightarrow S^{n+m}M\). 
If \(\mathsf C\) is furthermore cocomplete, the previous construction endows the coproduct
\[ S(M)\coloneqq\bigoplus_{n=0}^\infty S^nM \]
with a commutative algebra structure. This object is the free commutative monoid object over \(M\).
The required universality property follows by gluing together the universality of each symmetric power.

In the particular case where the underlying category is \(\mathsf{CMon}\), for a semimodule \(M\), \(S(M)\) is the free commutative semialgebra over \(M\).

\section{Quasi-shuffle via surjections}
\label{sec:quasiShuffle}

We remark that one can express the inductively defined quasi-shuffle product \eqref{Squasishuff} explicitly via certain surjections   \cite{EFM2017,EFMPW15}. Let $k,k_1,k_2$ be positive integers such that $\max(k_1,k_2) \le k \le k_1+k_2$.
We define a $(k_1,k_2)$-quasi-shuffle of type $k$ a surjection
$$
	f : \{1,\ldots, k_1+k_2 \} \twoheadrightarrow \{1,\ldots, k\},
$$
such that  $f(1) < \cdots < f({k_1})$ and $f({k_1+1}) < \cdots < f({k_1+k_2})$. Note that for $k=k_1+k_2$ one recovers the usual $(k_1,k_2)$-shuffle bijections. The set of $(k_1,k_2)$-quasi-shuffles of type $k$ is denoted by $\mathrm{qSh(k_1,k_2;k)}$. The latter permit to express the quasi-shuffle product \eqref{Squasishuff} of two words in closed form
 \begin{align}
 \label{surjections}
\begin{split}  
 & \left(a_{j_1} \cdots  a_{j_{k_1}}  \right)
	\qs 
   \left(a_{j_{k_1+1}} \cdots a_{j_{k_1+k_2}}	\right) 
   =
   \sum_{\max(k_1,k_2) \leq k \leq k_1+k_2} \;
   \sum_{f \in \mathrm{qSh(k_1,k_2;k)}} \;
         a^f_{i_1} \cdots a^f_{i_k},
  \end{split}
 \end{align} 
with $a^f_{i_l} \coloneqq {\prod_{m \in f^{-1}(\{l\})}} a_m$. Note that for $f \in \mathrm{qSh(k_1,k_2;k)}$ the set $ f^{-1}(\{l\})$ contains either one or two elements. In the case of a trivial bracket operation on $A$, the right-hand side of \eqref{surjections} reduces to the well-known formula expressing shuffle products of words in terms of shuffle permutations. Concretely, returning to Lemma \ref{lem:qsIdentity}, we see that \eqref{surjections} implies for two words over the alphabet \eqref{theAlphabet} and $z \in \mathbb{S}_{\zero_\s}^{d,\N_{\ge 1}}$
 \begin{align*}
   &\left\langle \ISS_{s,t}^\S(z), \left(a_{j_1} \cdots  a_{j_{k_1}}  \right)
	\qs 
   \left(a_{j_{k_1+1}} \cdots a_{j_{k_1+k_2}}	\right)  \right\rangle\\
    &=
    \sum_{\max(k_1,k_2) \leq k \leq k_1+k_2} \;
   \sum_{f \in \mathrm{qSh(k_1,k_2;k)}} \;
         \left\langle \ISS_{s,t}^\S(z), a^f_{i_1} \cdots a^f_{i_k}\right\rangle\\
& =
    \sum_{\max(k_1,k_2) \leq k \leq k_1+k_2} \;
   \sum_{f \in \mathrm{qSh(k_1,k_2;k)}} \;
   \sideset{}{_\s}\bigoplus_{s < j_1 < j_2 < \cdots < j_k\le t} z_{j_1}^{\odot_\s a^f_{i_1}} \odot_\s \dots \odot_\s z_{j_k}^{\odot_\s a^f_{i_k}}
  \end{align*}
where $a^f_{i_l} \coloneqq [{\prod_{m \in f^{-1}(\{l\})}} a_m]$.

\section*{Acknowledgments}
The first author thanks Bernd Sturmfels (MPI Leipzig) for introducing him to the tropical semiring.
The second author was supported by the Research Council of Norway through project 302831 ``Computational Dynamics and Stochastics on Manifolds'' (CODYSMA).
The third author was supported the BMS MATH+ Excellence Cluster EF1.
We thank Darij Grinberg (Drexel University) for very helpful comments on a first version of this manuscript.

\bibliographystyle{arxiv2}
\bibliography{tropical-iss.bib}
\end{document}